\documentclass[10pt,a4paper]{article}
\usepackage[latin1]{inputenc}
\usepackage{amsmath}
\usepackage{amsfonts}
\usepackage{amssymb}
\usepackage{amsthm}
\usepackage{graphicx}
\usepackage{enumitem}
\usepackage{epstopdf}

\usepackage{float} 

\usepackage{soul} 

\usepackage[a4paper,margin=2.0cm]{geometry}

\usepackage{caption}
\usepackage{natbib}
\usepackage{aas_macros}

\usepackage{hyperref}

\usepackage{pdflscape}	

\usepackage{siunitx}

\DeclareSIUnit\Msol{M_\odot}
\DeclareSIUnit\pc{pc}
\DeclareSIUnit\kpc{\kilo\pc}
\DeclareSIUnit\yr{yr}
\DeclareSIUnit\Myr{\mega\yr}
\DeclareSIUnit\Gyr{\giga\yr}
\DeclareSIUnit\AU{AU}

\newcommand*\diff{\mathop{}\!\mathrm{d}} 
\newcommand*\diffS{\mathop{}\!\mathrm{dS}} 
\newcommand*\R{\mathbb R} 
\newcommand*\N{\mathbb N} 
\newcommand*\measure{\mathcal L}

\DeclareMathOperator\supp{supp}
\DeclareMathOperator\divergence{div}
\DeclareMathOperator\rotation{rot}

\newtheorem{thm}{Theorem}[section]
\newtheorem{lem}[thm]{Lemma}
\newtheorem{prop}[thm]{Proposition}
\newtheorem{cor}[thm]{Corollary}

\theoremstyle{definition}
\newtheorem{defn}[thm]{Definition}
\theoremstyle{remark}
\newtheorem*{rem*}{Remark}
\newtheorem{rem}[thm]{Remark}

\numberwithin{equation}{section}

\title{\LARGE\textbf{A Mathematical Foundation for QUMOND}}
\author{Joachim Frenkler\\ Fakult\"at f\"ur Mathematik, Physik und Informatik\\Universit\"at Bayreuth \\ D-95440 Bayreuth, Germany \\ joachim.frenkler@uni-bayreuth.de}
\date{\today}

\graphicspath{{./}}
\usepackage{subfiles}

\begin{document}
	
	\maketitle
	
	\begin{abstract}
		We link the QUMOND theory with the Helmholtz-Weyl decomposition and introduce a new formula for the gradient of the Mondian potential using singular integral operators. This approach allows us to demonstrate that, under very general assumptions on the mass distribution, the Mondian potential is well-defined, once weakly differentiable, with its gradient given through the Helmholtz-Weyl decomposition. Furthermore, we establish that the gradient of the Mondian potential is an $L^p$ vector field. These findings lay the foundation for a rigorous mathematical analysis of various issues within the realm of QUMOND. Given that the Mondian potential satisfies a second-order partial differential equation, the question arises whether it has second-order derivatives. We affirmatively answer this question in the situation of spherical symmetry, although our investigation reveals that the regularity of the second derivatives is weaker than anticipated. We doubt that a similarly general regularity result can be proven without symmetry assumptions. In conclusion, we explore the implications of our results for numerous problems within the domain of QUMOND, thereby underlining their potential significance and applicability.
	\end{abstract}
	
	\section{Introduction}
	
	About 40 years ago \cite{1983ApJ...270..365Milgrom} proposed MOND (MOdified Newtonian Dynamics), a non-linear modification of Newton's law of gravity motivated by profound challenges in astrophysics. The basic MOND paradigm introduces a critical acceleration $a_0 = \SI{1.2e-10}{\m\per\square\s}$, stating that the real gravitational acceleration $g_{real}$ of an object and its acceleration $g_N$ expected from Newtonian gravity are related as follows:
	\begin{equation*} \label{intro basic MOND paradigm}
		\begin{array}{ll}
			g_{real} \approx \sqrt{a_0 g_N} & \text{if } g_N \ll a_0, \\
			g_{real} \approx g_N & \text{if } g_N \gg a_0. \nonumber
		\end{array}
	\end{equation*}
	Thus, in the regime of large accelerations, MOND predicts behaviour consistent with Newtonian gravity. However, at extremely low accelerations MOND predicts that $g_{real}$ is proportional to the square root of the acceleration $g_N$ expected from Newtonian physics. With its single modification MOND offers explanations for many astrophysical phenomena \citep{2012LRR....15...10FamaeyMcGaugh}, most notably flat rotation curves \citep{2011A&A...527A..76GentileFamaeyBlok}. While MOND very effectively describes dynamics on the scales of galaxies, it faces more serious problems on scales slightly larger than the solar system. Recent debates have emerged regarding whether the data from GAIA on wide binary stars supports MOND \citep{2023ApJ...952..128Chae,2023arXiv231203162HernandezChae} or contradicts it \citep{2024MNRAS.527.4573Banik}.
	
	The present paper focuses on mathematical questions, analysing the equations that are used to describe Mondian physics. We study in detail whether these equations are well posed, introduce a new formula for the Mondian gravitational field using the Helmholtz-Weyl decomposition, and analyse the regularity of the Mondian potential and its derivatives.
	
	When considering to replace Newton's law of gravity by Mondian gravity, one is tempted, in view of the basic MOND paradigm, to simply replace the Newtonian field $\nabla U^N_\rho$, which corresponds to some density $\rho$ on $\R^3$, by
	\begin{equation} \label{intro MOND field first ansatz}
	\nabla U^N_\rho + \lambda\left(\left| \nabla U^N_\rho \right|\right) \nabla U^N_\rho
	\end{equation}
	where $\lambda:[0,\infty)\rightarrow[0,\infty)$ is such that
	\begin{equation*}
	\begin{array}{ll}
	\lambda(u) \approx \sqrt{a_0}/\sqrt{u} & \text{if } u \ll a_0, \\
	\lambda(u) \approx 0 & \text{if } u \gg a_0.
	\end{array}
	\end{equation*}
	But then one runs into a problem. The field \eqref{intro MOND field first ansatz} is in general not the gradient of some potential, thus leading to a loss of classical conservation laws of physics like conservation of momentum \cite[§6]{2012LRR....15...10FamaeyMcGaugh}. A more refined approach is required. \cite{2010MNRAS.403..886Milgrom} proposed a theory, which is called QUMOND (QUasi linear formulation of MOND), where the Mondian potential $U^M_\rho$ is defined as the solution of the partial differential equation (PDE)
	\begin{equation} \label{intro MOND field via QUMOND PDE}
	\divergence\left( \nabla U^M_\rho\right) = \divergence \left( \nabla U^N_\rho + \lambda\left(\left| \nabla U^N_\rho \right|\right) \nabla U^N_\rho \right).
	\end{equation}
	
	Is the above PDE well posed? \cite{2010MNRAS.403..886Milgrom} provided an explicit formula for its solution $U^M_\rho$. But is this $U^M_\rho$ well defined? And if yes, which regularity properties does it have? These questions we answer in the present paper. To do so we develope a new mathematical foundation for the QUMOND theory using the Helmholtz-Weyl decomposition. Simply put, the Helmholtz-Weyl decomposition states that every well-behaved vector field $v$ that vanishes at infinity can be uniquely decomposed into an irrotational vector field plus a solenoidal vector field. While the solenoidal field has a vector potential, the irrotational field has a scalar potential $U$, and $U$ satisfies the PDE
	\begin{equation} \label{intro scalar potential PDE}
	\divergence\left( \nabla U\right) = \divergence \left( v \right).
	\end{equation}
	Comparing the PDEs \eqref{intro MOND field via QUMOND PDE} and \eqref{intro scalar potential PDE}, we see that we can identify $v$ with the vector field \eqref{intro MOND field first ansatz} and the potential $U$ with $U^M_\rho$. Thus, the vector field $\nabla U^M_\rho$ should be the irrotational part of the vector field \eqref{intro MOND field first ansatz}. 
	
	In this paper we use the Helmholtz-Weyl decomposition in the form proven by \cite{2011Galdi} for $L^p$ vector fields\footnote{As usual, we say that some function $f$ is an $L^p$ function, i.e., $f\in L^p(\R^3)$ for some $1<p<\infty$, if $\int |f(x)|^p \diff x < \infty$. We say that a vector field $v\in L^p(\R^3)$ if all three components are $L^p$ functions.} and introduce a new explicit formula for the irrotational part of a vector field on $\R^3$ using singular integral operators. These operators are used to derive a new, explicit expression for the Mondian gravitational field $\nabla U^M_\rho$ too. This new formulation is very useful to analyse the PDE \eqref{intro MOND field via QUMOND PDE} and the regularity properties of $\nabla U^M_\rho$. It enables us to prove the following theorem:
	
	\begin{thm} \label{thm main}
		For every density $\rho$ on $\R^3$ that has finite mass and is an $L^p$ function for some $p>1$, the corresponding Mondian potential $U^M_\rho$ -- defined as in \cite{2010MNRAS.403..886Milgrom} -- is well defined and once weakly differentiable with $\nabla U^M_\rho$ being the irrotational part of the vector field  \eqref{intro MOND field first ansatz} in the sense of the Helmholtz-Weyl decomposition. $\nabla U^M_\rho= \nabla U^N_\rho + \nabla U^\lambda_\rho$ can be decomposed into an $L^q$ vector field $\nabla U^N_\rho$ plus an $L^r$ vector field $\nabla U^\lambda_\rho$ with $q>3/2$ and $r>3$. The potential $U^M_\rho$ solves the PDE \eqref{intro MOND field via QUMOND PDE} in distribution sense
	\end{thm}
	
	Further in this paper, we analyse second derivatives of $U^M_\rho$. Under the additional assumptions that $\rho$ is bounded and spherically symmetric, we prove that $U^M_\rho$ is twice weakly differentiable and $D^2 U^M_\rho$ is an $L^r$ function. Using handwaving arguments, one would expect that this should hold for $1<r<6$. But this is wrong. It is only possible to prove that $D^2 U^M_\rho \in L^r(\R^3)$ for $1<r<2$ and this result is really optimal. Through counterexamples, we show that it is impossible to achieve such a regularity result for $r>2$. This is a surprising fact and it is due to the square root appearing above in the basic MOND paradigm. We discuss why achieving similarly general regularity results for the second derivatives of $U^M_\rho$ without assuming spherical symmetry seems doubtful.
	
	The regularity results for $U^M_\rho$, $\nabla U^M_\rho$ and $D^2 U^M_\rho$ presented in this paper are essential for addressing further important questions. For example they enable us to examine whether initial value problems using Mondian gravity are well-posed, whether corresponding solutions conserve energy, or what the stability properties of stationary solutions are. The present paper forms the foundation for treating these questions with mathematical rigour.
	
	The outline of this paper is as follows. In Section \ref{section Newtonian potentials}, we analyse Newtonian potentials, a prerequisite for analysing Mondian potentials later on, and we introduce the singular integral operators that are important for the rest of the paper. In Section \ref{section irrotational vector fields}, we study the Helmholtz-Weyl theory from \cite{2011Galdi} and provide a new expression for the irrotational part of a vector field using the singular integral operators defined previously. In Section \ref{section Mondian potentials}, we bring together the QUMOND theory from \cite{2010MNRAS.403..886Milgrom} and our new knowledge about the Helmholtz-Weyl decomposition to prove Theorem \ref{thm main}. In Section \ref{section second derivatives}, we analyse the (non-)existence of second derivatives of the Mondian potential. In Section \ref{section discussion}, we discuss how the results of this paper can be applied to many problems in QUMOND.
	
	\section{Newtonian potentials} \label{section Newtonian potentials}

	In this paper Newtonian potentials will play an important role in two different ways. On the one hand when we have a certain mass distribution with density $\rho$ then $U^N_\rho$ is the Newtonian gravitational potential that belongs to the density $\rho$. On the other hand in the QUMOND theory we must understand how to decompose a vector field $v$ in its irrotational and its solenoidal part. Here the Newtonian potentials of the three components $v_i$ of the vector field play an important role. This we treat in Section \ref{section irrotational vector fields}.
	
	Given a density $\rho = \rho(x)$, $x\in\R^3$, then the corresponding Newtonian gravitational potential $U^N_\rho$ is given by
	\begin{equation} \label{VQMS Newtonian potential}
	U^N_\rho(x) = - G\int \frac{\rho(y)}{|x-y|} \diff y, \quad x\in\R^3,
	\end{equation}
	provided the convolution integral exists. Since the concrete value of the gravitational constant $G$ does not affect our analysis we set it to unity. Next we want to introduce some useful singular integral operators. For $\epsilon>0$, $i,j=1,2,3$ and a measurable function $g:\R^3\rightarrow\R$ we define
	\begin{equation*}
	T_{ij}^\epsilon g(x) := - \int_{|x-y|>\epsilon} \left[ 3\frac{(x_i-y_i)(x_j-y_j)}{|x-y|^5} - \frac{\delta_{ij}}{|x-y|^3} \right] g(y) \diff y  , \quad x\in\R^3,
	\end{equation*}
	provided that the convolution integral on the right hand side exists. Since
	\begin{equation*}
	\partial_{x_i}\partial_{x_j} \frac{1}{|x|} = 3\frac{x_ix_j}{|x|^5} - \frac{\delta_{ij}}{|x|^3},
	\end{equation*}
	the limit of $T_{ij}^\epsilon g$ for $\epsilon\rightarrow 0$ plays an important role in understanding the second derivatives of the Newtonian potential $U^N_g$. Further, it plays an important role for the Helmholtz-Weyl decomposition as we will see below and hence for the QUMOND theory.  In the following two propositions we study this limit.
	
	\begin{prop} \label{prop regularity of Tijg for g continuous}
		For every $\epsilon>0$ and $g\in C_c^1(\R^3)$
		\begin{equation*}
		T_{ij}^\epsilon g \in C(\R^3)
		\end{equation*}
		and the limit
		\begin{equation*}
		T_{ij} g := \lim_{\epsilon\rightarrow 0} T_{ij}^\epsilon g
		\end{equation*}
		exists in $L^\infty(\R^3)$. In particular 
		\begin{equation*}
		T_{ij}g \in C(\R^3).
		\end{equation*}
		
	\end{prop}
	
	\begin{prop} \label{prop regularity of Tijg for g in Lp}
		Let $1<p<\infty$. There is a $C_p>0$ such that for every $\epsilon>0$ and $g\in L^p(\R^3)$
		\begin{equation*}
		\|T_{ij}^\epsilon g\|_p \leq C_p \|g\|_p
		\end{equation*}
		and the limit
		\begin{equation*}
		T_{ij}g := \lim_{\epsilon\rightarrow 0} T_{ij}^\epsilon g
		\end{equation*}
		exists in $L^p(\R^3)$ with
		\begin{equation*}
		\|T_{ij}g\|_p \leq C_p \|g\|_p.
		\end{equation*}
	\end{prop}

	\begin{proof}[Proof of Proposition \ref{prop regularity of Tijg for g continuous} and \ref{prop regularity of Tijg for g in Lp}]
		The statements follow quite directly from the literature. To apply the results from the literature, we have to verify that
		\begin{equation*}
		\Omega_{ij}(x) := 3\frac{x_ix_j}{|x|^2} - \delta_{ij}, \quad x\in\R^3, x\neq 0,
		\end{equation*}
		satisfies the following four assumptions:
		\begin{enumerate}
			\item $\Omega_{ij}$ must be homogeneous of degree 0, i.e., $\Omega_{ij}(\delta x) = \Omega_{ij}(x)$ for all $\delta>0$, $x\neq 0$. This is obviously true.
			
			\item $\Omega_{ij}$ must satisfy the cancellation property
			\begin{equation*}
			\int_{|x|=1} \Omega_{ij}(x) \diffS(x) = 0.
			\end{equation*}
			If $i\neq j$, this is obviously true. If $i=j$ this is also true, since
			\begin{align*}
			\int_{|x|=1} \Omega_{ii}(x) \diffS(x) & = 3 \int_{|x|=1} x_i^2 \diffS(x) - 4\pi \\
			& = \int_{|x|=1} |x|^2 \diffS(x) - 4\pi = 0.
			\end{align*}
			
			\item $\Omega_{ij}$ must be bounded on $\{|x|=1\}$. This is obviously true since $\Omega_{ij}$ is continuous on $\R^3\backslash\{0\}$.
			
			\item $\Omega_{ij}$ must satisfy the following smoothness property: For 
			\begin{equation*}
			w(\delta) := \sup\limits_{\substack{|x-x'|<\delta \\ |x|=|x'|=1}} |\Omega_{ij}(x) - \Omega_{ij}(x')|
			\end{equation*}
			must hold
			\begin{equation*}
			\int_0^1 \frac{w(\delta)}{\delta} \diff \delta < \infty.
			\end{equation*}
			This is true since for $x,x'\in\R^3$ with $|x|=|x'|=1$ and $|x-x'|<\delta$ we have
			\begin{equation*}
			|\Omega_{ij}(x) - \Omega_{ij}(x')| = 3|x_ix_j - x'_ix'_j| \leq 3|x_i||x_j-x_j'| + 3|x_j'||x_i-x_i'| \leq 6\delta.
			\end{equation*}
		\end{enumerate}
		Now Proposition \ref{prop regularity of Tijg for g in Lp} follows directly from \cite[Chapter II, Theorem 3]{1970Stein} and Proposition \ref{prop regularity of Tijg for g continuous} follows from \cite[Satz 2.2]{2001Dietz}. In the formulation of her theorem \citeauthor{2001Dietz} does not mention the continuity of the $T_{ij}^\epsilon g$, but studying her proof carefully one sees that she has proven the H\"older continuity of $T_{ij}^\epsilon g$ under the assumption that $\supp g\subset B_1$. This holds obviously also for every $g\in C_c^1(\R^3)$ after a suitable scaling. If however one is interested solely in the continuity of $T_{ij}^\epsilon g$, like we here in this paper, one could also simply apply the transformation $y\mapsto x-y$ in the definition of $T_{ij}^\epsilon g$ and use standard results to deduce that $T_{ij}^\epsilon g$ is continuous.
		
	\end{proof}
	
	Next we formulate regularity results for the Newtonian potential. Note that we have set the gravitational constant $G$ to unity. 
	
	\begin{lem} \label{lemma Newtonian potential for smooth density}
		Let $g\in C^{1+n}_c(\R^3)$, $n\in\N_0$. Then the following holds
		\begin{enumerate}[label=\alph*)]
			\item \label{lemma Newtonian potential for smooth density - Formulae}
			The Newtonian potential $U_g^N\in C^{2+n}(\R^3)$. Its first derivative is given by 
			\begin{equation*}
			\partial_{x_i} U^N_g = U^N_{\partial_{y_i}g}, \quad i=1,2,3,
			\end{equation*}
			which, using integration by parts, can be written as 
			\begin{equation*}
			\nabla U_g ^N (x) = \int \frac{x-y}{|x-y|^3} g(y) \diff y, \quad x\in\R^3.
			\end{equation*}
			The second derivative of $U^N_g$ is given by
			\begin{equation*}
			\partial_{x_i}\partial_{x_j} U^N_g = T_{ij}g + \delta_{ij}\frac{4\pi}{3}g,
			\end{equation*}
			where $i,j=1,2,3$.
			
			\item \label{lemma Newtonian potential for smooth density - Continuity}
			For every $R>0$ there is a $C>0$ such that
			\begin{equation*}
			\|U^N_g\|_\infty + \|\nabla U^N_g\|_\infty \leq C \|g\|_\infty.
			\end{equation*}
			and
			\begin{equation*}
			\|D^2 U^N_g\|_\infty \leq C(\|g\|_\infty + \|\nabla g\|_\infty)
			\end{equation*}
			provided $\supp g\subset B_R$.
			\item \label{lemma Newtonian potential for smooth density - PDE}
			$U^N_g$ is the unique solution of
			\begin{equation*}
			\Delta U^N_g = 4\pi g, \quad \lim_{|x|\rightarrow\infty} U^N_g(x) = 0.
			\end{equation*}
			in $C^2(\R^3)$
		\end{enumerate}
	\end{lem}
	
	\begin{proof}
		It is proven in  \cite[Lemma P1]{2007Rein} that $U^N_g\in C^2(\R^3)$ if $g\in C_c^1(\R^3)$ and that the formulae for the first derivatives hold. 
		If $g\in C^{1+n}_c(\R^3)$ with $n \geq 1$, it follows directly from
		\begin{equation*}
		\partial_{x_i} U^N_g = U^N_{\partial_{y_i}g}
		\end{equation*}		
		that $U^N_g\in C^{2+n}(\R^3)$.
		To prove \ref{lemma Newtonian potential for smooth density - Formulae} it remains to verify the formula for the second derivatives. For every $x\in\R^3$ we have
		\begin{align*}
		\partial_{x_i}\partial_{x_j} U^N_g (x)  & = \partial_{x_i}  U^N_{\partial_{y_j}g}(x) = \int\frac{x_i-y_i}{|x-y|^3} \partial_{y_j}g(y) \diff y = \\
		& = - \int \frac{y_i}{|y|^3} \partial_{y_j}(g(x-y)) \diff y = \int \partial_{y_i}(|y|^{-1}) \partial_{y_j}(g(x-y)) \diff y.
		\end{align*}
		Dominated convergences and integration by parts then yield
		\begin{align*}
		\partial_{x_i}\partial_{x_j} U^N_g (x) & =  \lim_{\epsilon\rightarrow 0} \int_{|y|>\epsilon}\partial_{y_i}(|y|^{-1})\partial_{y_j} (g(x-y)) \diff y \\
		& = \lim_{\epsilon\rightarrow 0} \left(T_{ij}^\epsilon g (x) + \int_{|y|=\epsilon} \frac{y_iy_j}{|y|^4} g(x-y) \diffS(y)\right);
		\end{align*}
		observe that the normal on $\{|y|=\epsilon\}$ is pointing inward and that there is no border term at infinity due to the compact support of $g$. $T_{ij}^\epsilon g$ converges uniformly to $T_{ij}g$ after Proposition \ref{prop regularity of Tijg for g continuous}. If $i\neq j$ then
		\begin{align*}
		\left|\int_{|y|=\epsilon} \frac{y_iy_j}{|y|^4} g(x-y) \diffS(y)\right| & = \left|\int_{|y|=\epsilon} \frac{y_iy_j}{|y|^4} (g(x-y) - g(x)) \diffS(y)\right| \leq 4\pi\|\nabla g\|_\infty\epsilon.
		\end{align*}
		Hence the border term vanishes. If $i=j$ then
		\begin{align*}
		\int_{|y|=\epsilon} &\frac{y_i^2}{|y|^4}  g(x-y) \diffS(y) = \int_{|y|=\epsilon} \frac{y_i^2}{|y|^4} (g(x-y)-g(x)) \diffS(y) + g(x)\int_{|y|=\epsilon} \frac{y_i^2}{|y|^4} \diffS(y).
		\end{align*}
		As above the first term vanishes, however, the second one evaluates to $4\pi g(x)/3$. In total we get
		\begin{equation*}
		\partial_{x_i}\partial_{x_j} U^N_g (x)  = T_{ij}g(x) + \delta_{ij}\frac{4\pi}{3}g(x).
		\end{equation*}
		
		Let us turn to \ref{lemma Newtonian potential for smooth density - Continuity}. Since $\supp g\subset B_R$ and $g$ is bounded one sees directly that
		\begin{equation*}
		\|U^N_g\|_\infty + \|\nabla U^N_g\|_\infty \leq C \|g\|_\infty.
		\end{equation*}
		That
		\begin{equation*}
		\|D^2 U^N_g\|_\infty \leq C(\|g\|_\infty + \|\nabla g\|_\infty),
		\end{equation*}
		is proven in \cite[Lemma P1]{2007Rein}.
		
		It remains to show \ref{lemma Newtonian potential for smooth density - PDE}. It is stated in  \cite[Lemma P1]{2007Rein} that $U^N_g$ is the unique solution of 
		\begin{equation*}
		\Delta U^N_g = 4\pi g, \quad \lim_{|x|\rightarrow\infty} U^N_g(x) = 0,
		\end{equation*}
		however the proof is omitted. So let us briefly summarize the proof of this well known fact. Since
		\begin{equation*}
		\sum_{i=1}^3 \left( 3\frac{x_i^2}{|x|^5} - \frac{1}{|x|^3} \right) = 0,
		\end{equation*}
		we have
		\begin{equation*}
		\sum_{i=1}^3 T_{ii}g = 0.
		\end{equation*}
		Thus
		\begin{equation*}
		\Delta U^N_g = \sum_{i=1}^3 \partial_{x_i}^2 U^N_g = \sum_{i=1}^3 \left( T_{ii}g + \frac{4\pi}{3}g\right) = 4\pi g.
		\end{equation*}
		The asymptotic behaviour of $U^N_g(x)$ for $|x|\rightarrow\infty$ follows from the compact support of $g$. That $g$ is the unique solution of the above PDE follows from the strong maximum principle \cite[Theorem 2.2.]{1977Gilbarg}.
	\end{proof}
	
	\begin{lem} \label{lemma Newtonian potential for density in Lp NEW}
		Let $g\in L^1\cap L^p(\R^3)$ for a $1<p<\infty$. Then $U^N_g\in L^1_{loc}(\R^3)$ exists, is twice weakly differentiable and the formulae for $\nabla U^N_g$ and $\partial_{x_i}\partial_{x_j} U^N_g$ from Lemma \ref{lemma Newtonian potential for smooth density} and the following estimates hold
		\begin{enumerate}[label=\alph*)]
			\item If $1<p<\frac{3}{2}$ and $3<r<\infty$ with $\frac{1}{3} + \frac{1}{p} = 1 + \frac{1}{r}$ then
			\begin{equation*}
			\|U^N_g\|_r \leq C_{p,r} \|g\|_p.
			\end{equation*}
			\item If $1<p<3$ and $\frac{3}{2} < s < \infty $ with $\frac{2}{3} + \frac{1}{p} = 1 + \frac{1}{s}$ then
			\begin{equation*}
			\|\nabla U^N_g\|_s \leq C_{p,s} \|g\|_p.
			\end{equation*}
			\item For every $1<p<\infty$
			\begin{equation*}
			\|D^2U^N_g\|_p \leq C_p \|g\|_p.
			\end{equation*}
		\end{enumerate}
	\end{lem}
	
	%
	
	\begin{proof}
		With the formula for $\nabla U^N_g$ as in Lemma \ref{lemma Newtonian potential for smooth density} we have
		\begin{equation*}
		U^N_g = - \frac{1}{|\cdot|} * g \quad \text{and} \quad
		\nabla U^N_g = \frac{\cdot}{|\cdot|^3} * g.
		\end{equation*}
		$1/|\cdot|$ and $\cdot/|\cdot|^3$ are in the so called weak $L^q$-space with $q=3$ and $q=\frac{3}{2}$ respectively since
		\begin{equation*}
		\sup_{\alpha>0} \alpha \measure\left(\left\{x:\frac{1}{|x|}>\alpha\right\}\right)^{1/3} = (4\pi/3)^{1/3} < \infty
		\end{equation*}
		and
		\begin{equation*}
		\sup_{\alpha>0} \alpha \measure\left(\left\{x:\frac{1}{|x|^2}>\alpha\right\}\right)^{2/3} = (4\pi/3)^{2/3} < \infty;
		\end{equation*}
		with $\measure(\Omega)$ we denote the Lebesgue measure of a measurable set $\Omega\subset\R^n$, $n\in\N$. Thus \cite[Remark 4.3(2)]{2010LiebLoss} implies that $U^N_g\in L^r$ and $\nabla U^N_g \in L^s$ with the desired estimates provided $p<3/2$ and $p<3$ respectively. If $p\geq 3/2$, $U^N_\rho\in L^r(\R^3)$ for every $3<r<\infty$ since $\rho\in L^1\cap L^p(\R^3) \subset L^q(\R^3)$ for every $1<q<3/2$. The same argumentation holds for $\nabla U^N_\rho$ if $p\geq 3$.
		
		We have to check that $\nabla U^N_g$ is indeed the weak derivative of $U_g^N$. For this take $\phi\in C_c^\infty(\R^3)$. The Hardy-Littlewood-Sobolev inequality \cite[Theorem 4.3]{2010LiebLoss} allows us to use Fubini:
		\begin{equation*}
		\int U^N_g(x) \partial_{x_i}\phi(x) \diff x =
		-\iint \frac{g(y)\partial_{x_i}\phi(x)}{|x-y|} \diff x \diff y.
		\end{equation*}
		Now Lemma \ref{lemma Newtonian potential for smooth density} implies
		\begin{align*}
		\int U^N_g(x) \partial_{x_i}\phi(x) \diff x & = \int g(y) U^N_{\partial_{x_i}\phi}(y) \diff y = \int g(y) \partial_{y_i}U^N_{\phi}(y) \diff y \\
		& = \iint g(y) \frac{y_i-x_i}{|y-x|^3} \phi(x) \diff x \diff y \\
		& = - \int \left(\int\frac{x_i-y_i}{|x-y|} g(y) \diff y\right) \phi(x) \diff x \\
		& = - \int \nabla U^N_g(x) \phi(x) \diff x.
		\end{align*}
		So the  weak gradient of $U^N_g$ is given by the formula for $\nabla U^N_g$ from Lemma \ref{lemma Newtonian potential for smooth density}.
		
		Let $1<p<\infty$. We study the second derivatives and take $(g_k)\subset C_c^1(\R^3)$ such that
		\begin{equation*}
		g_k\rightarrow g \quad \text{in } L^p(\R^3) \text{ for } k\rightarrow\infty.
		\end{equation*}
		Then H\"older, integration by parts and Lemma \ref{lemma Newtonian potential for smooth density} give
		\begin{align*}
		\int U^N_g \partial_{x_i}\partial_{x_j} \phi \diff x & = \lim_{k\rightarrow\infty} \int U^N_{g_k} \partial_{x_i}\partial_{x_j} \phi \diff x \\
		& = \lim_{k\rightarrow\infty} \int (T_{ij} g_k + \delta_{ij}\frac{4\pi}{3}g_k ) \phi \diff x \\
		& = \int (T_{ij} g + \delta_{ij}\frac{4\pi}{3}g ) \phi \diff x.
		\end{align*}
		Thus the weak second derivatives of $U^N_g$ are given by the same formula as in Lemma \ref{lemma Newtonian potential for smooth density}. The desired estimate for $\partial_{x_i}\partial_{x_j}U^N_g$ follows from Proposition \ref{prop regularity of Tijg for g in Lp}.
		
	\end{proof}

	In the situation of spherical symmetry there is a second formula for the Newtonian field $\nabla U^N_\rho$, which often is quite useful.
	
	\begin{lem} \label{lemma Newtons shell theorem}
		Let $1<p<3$ and $\rho\in L^1\cap L^p(\R^3)$, $\geq 0$ be spherically symmetric. Then
		\begin{equation*}
		\nabla U^N_\rho(x) = \frac{M(r)}{r^2} \frac{x}{r}
		\end{equation*}
		for a.e. $x\in\R^3$ with $r= |x|$ and
		\begin{equation*}
		M(r) := \int_{B_r} \rho(x) \diff x = 4\pi \int_0^r s^2 \rho(s) \diff s
		\end{equation*}
		denoting the mass inside the ball with radius $r$.
	\end{lem}
	
	\begin{proof}
		This lemma was first proven by \cite{1687pnpm.book.....Newton} in a similar version. Below we give a proof of our modern version using $L^p$ theory.
		
		Assume that $\rho$ would be continuous and compactly supported. Then $M\in C^1([0,\infty))$ with
		\begin{equation*}
		M'(r) = 4\pi r^2 \rho(r), \quad r\geq 0.
		\end{equation*}
		Further
		\begin{equation*}
		|M(r)| \leq \|\rho\|_1
		\end{equation*}
		and
		\begin{equation*}
		|M(r)| \leq \frac{4\pi}{3} \|\rho\|_\infty r^3
		\end{equation*}
		for $r\geq 0$. For $x\in\R^3$ and $r=|x|$
		\begin{equation*}
		U(x)  := - \int_r^\infty \frac{M(s)}{s^2}\diff s
		\end{equation*}
		is well defined. If $r=|x|>0$, $U$ is continuously differentiable with
		\begin{equation*}
		\nabla U(x) = \frac{M(r)}{r^2} \frac xr.
		\end{equation*}
		Since
		\begin{equation*}
		|\nabla U(x)| \leq \frac{4\pi}{3} \|\rho\|_\infty r
		\end{equation*}
		we have
		\begin{equation*}
		\nabla U \in C(\R^3).
		\end{equation*}
		Further
		\begin{equation*}
		\partial_{x_i}\partial_{x_j} U(x) = 4\pi\rho(r)\frac{x_ix_j}{r^2} - 3M(r)\frac{x_ix_j}{r^5} + \frac{M(r)}{r^3} \delta_{ij}, \quad r>0.
		\end{equation*}
		Since $\rho$ is continuous,
		\begin{equation*}
		\frac{M(r)}{r^3} = \frac{4\pi}{3} \frac{1}{\measure(B_r)} \int_{B_r} \rho\diff x \rightarrow \frac{4\pi}{3} \rho(0)
		\end{equation*}
		and
		\begin{equation*}
		\left| 4\pi\rho(r) - \frac{3M(r)}{r^3} \right| = 4\pi \left| \rho(r) - \frac{1}{\measure(B_r)} \int_{B_r}\rho\diff x \right| \rightarrow 0
		\end{equation*}
		for $r\rightarrow 0$. Hence
		\begin{equation*}
		D^2U \in C(\R^3).
		\end{equation*}
		Thus
		\begin{equation*}
		U\in C^2(\R^3).
		\end{equation*}
		Further
		\begin{equation*}
		\Delta U = 4\pi \rho
		\end{equation*}
		and
		\begin{equation*}
		\lim_{|x|\rightarrow\infty}|U(x)| \leq \lim_{|x|\rightarrow\infty} \frac{\|\rho\|_1}{|x|} = 0.
		\end{equation*}
		Since by Lemma \ref{lemma Newtonian potential for smooth density} $U^N_\rho$ is a solution of this PDE, too, and this solutions is unique 
		\begin{equation*}
		U^N_\rho = U
		\end{equation*}
		and
		\begin{equation} \label{equ Newtons shell theorem nabla U for smooth rho}
		\nabla U^N_\rho(x) = \nabla U(x) = \frac{M(r)}{r^2} \frac{x}{r}, \quad x\in\R^3.
		\end{equation}
		
		If now $\rho\in L^1\cap L^p(\R^3)$, we take a sequence $(\rho_n)\subset C_c(\R^3)$ of spherically symmetric densities such that
		\begin{equation*}
		\rho_n \rightarrow \rho \quad \text{in }L^p(\R^3)\text{ for }n\rightarrow\infty.
		\end{equation*}
		By Lemma \ref{lemma Newtonian potential for density in Lp NEW} 
		\begin{equation} \label{equ Newtons shell theorem nabla Un converges to nabla U}
		\nabla U^N_{\rho_n} \rightarrow \nabla U^N_\rho \quad \text{in } L^s(\R^3) \text{ for } n\rightarrow\infty
		\end{equation}
		where $s>3/2$ with $1/p + 2/3 = 1 + 1/s$. Set
		\begin{equation*}
		M_n(r) := \int_{B_r} \rho_n \diff x, \quad r\geq 0.
		\end{equation*}
		Then for every $r\geq 0$
		\begin{equation*}
		|M_n(r)-M(r)| \leq \| \rho_n - \rho \|_p \|1_{B_r}\|_{p/(p-1)}.
		\end{equation*}
		Hence for all $0<S<R$
		\begin{equation*}
		M_n \rightarrow M \quad \text{uniformly on }B_R\text{ for }n\rightarrow\infty
		\end{equation*}
		and
		\begin{equation*}
		\frac{M_n(r)}{r^2} \frac xr \rightarrow \frac{M(r)}{r^2} \frac xr \text{ uniformly on } \{S<|x|<R\} \text{ for } n\rightarrow\infty.
		\end{equation*}
		Together with \eqref{equ Newtons shell theorem nabla U for smooth rho} and \eqref{equ Newtons shell theorem nabla Un converges to nabla U} this implies that for a.e. $x\in\R^3$
		\begin{equation*}
		\nabla U^N_\rho(x) = \frac{M(r)}{r^2} \frac xr.
		\end{equation*}
	\end{proof}
	
	Later on, we will make regular use of the following statement.
	
	\begin{lem} \label{lemma EpotN for rho in L6/5}
		If $\rho,\sigma\in L^{6/5}(\R^3)$ then
		\begin{equation*}
		-\frac{1}{8\pi} \int \nabla U^N_\rho\cdot \nabla U^N_\sigma \diff x = \frac 12 \int U^N_\rho\sigma\diff x = -\frac 12 \iint \frac{\rho(y)\sigma(x)}{|x-y|} \diff x\diff y.
		\end{equation*}
	\end{lem}
	
	\begin{proof}
		$\nabla U^N_\rho,\nabla U^N_\sigma\in L^{2}(\R^3)$ and $U_\rho^N\in L^6(\R^3)$ according to Lemma \ref{lemma Newtonian potential for density in Lp NEW}. Thus the first two integrals are well defined. By the Hardy-Littlewood-Sobolev inequality \citep[Theorem 4.3]{2010LiebLoss} also the third integral is well defined. If $\rho,\sigma \in C_c^\infty(\R^3)$, integration by parts and $\Delta U^N_\sigma = 4\pi\sigma$ give the above equalities of the integrals. Since $C_c^\infty(\R^3)\subset L^{6/5}(\R^3)$ is dense and all three integrals above are continuous maps from $L^{6/5}\times L^{6/5}\rightarrow \R$, the above equalities hold for all $\sigma,\rho\in L^{6/5}(\R^3)$.
	\end{proof}

	\section{Irrotational vector fields} \label{section irrotational vector fields}
	
	As stated in Theorem \ref{thm main}, we want to prove that the gradient of the Mondian potential $U^M_\rho$ is the irrotational part of the vector field $\nabla U^N_\rho + \lambda(|\nabla U^N_\rho|)\nabla U^N_\rho$. In this section we specify what we mean with the `irrotational part of a vector field'. To do so, we make use of the singular integral operators $T_{ij}$ introduced in the previous section about Newtonian potentials.
	
	\begin{defn} \label{definition Helmholtz operator H}
		Let $1<p<\infty$ and $v\in L^p(\R^3)$ be a vector field. For $i=1,2,3$ we define 
		\begin{equation*}
		H_i v := \frac{1}{4\pi}\sum_{j=1}^{3} T_{ij}v_j + \frac{1}{3} v_i.
		\end{equation*}
		We call the vector field $Hv\in L^p(\R^3)$ the irrotational part of $v$.
	\end{defn}

	We will see below that $Hv$ is indeed the irrotational part of $v$ in the sense of the the Helmholtz-Weyl decomposition:
	
	\begin{thm}[Helmholtz-Weyl decomposition] \label{thm Helmholtz-Weyl decompostion}
		For every vector field $v\in L^p(\R^3)$, $1<p<\infty$, exist uniquely determined $v_1\in L^p_{irr}(\R^3)$ and $v_2\in L^p_{sol}(\R^3)$ such that
		\begin{equation*}
		v = v_1+v_2,
		\end{equation*}
		where the two subspaces $L^p_{irr}(\R^3)$ and $L^p_{sol}(\R^3)$ of $L^p(\R^3)$ are defined as follows:
		\begin{align*}
		L^p_{irr}(\R^3) & := \left\{ w\in L^p(\R^3) \text{ such that } U\in W^{1,p}_{loc}(\R^3) \text{ exists with } w = \nabla U  \right\},\\
		L^p_{sol}(\R^3) & := \left\{  w\in L^p(\R^3) \text{ such that } \divergence w = 0 \text{ weakly} \right\}.
		\end{align*}
	\end{thm}

	\begin{rem*}
		The space $W^{1,p}_{loc}(\R^3)$ denotes the Sobolev space of scalar functions $U$ on $\R^3$ that are once weakly differentiable and that are locally integrable if taken to the power $p$, i.e., for each compact domain $K$ the integral $\int_K |U|^p \diff x$ is finite. Further, also the gradient of $U$ must be locally integrable if taken to the power $p$, but observe that in the definition of $L^p_{irr}$ we additionally demanded that the gradient shall be an $L^p$ vector field not only on every compact domain but on the entire space $\R^3$.
	\end{rem*}
	
	\begin{proof}
		In \cite[Theorem III.1.2]{2011Galdi} it is proven that the Helmholtz-Weyl decomposition in the sense of \cite[equation (III.1.5)]{2011Galdi} holds. This form of the theorem makes use of a different definition of the space $L^p_{sol}$. However, in \cite[Theorem III.2.3]{2011Galdi} it is proven that our definition here coincides with the definition used in \cite[Theorem III.1.2]{2011Galdi}.
	\end{proof}
	
	Let us study how the Helmholtz-Weyl decomposition looks like for smooth vector fields with compact support.
	
	\begin{lem} \label{lemma Helmholtz-Weyl decompostion for smooth v}
		Let $v\in C_c^2(\R^3)$. For every $1<p<\infty$
		\begin{equation*}
		Hv = \frac{1}{4\pi}\nabla \left( \sum_{j=1}^{3} \partial_{x_j} U^N_{v_j} \right) \in C^1 \cap L^p_{irr}(\R^3)
		\end{equation*}
		is the uniquely determined, irrotational part of the vector field $v$ according to the Helmholtz-Weyl decomposition. Further
		\begin{equation*}
		\divergence Hv = \divergence v \text{ and } \rotation Hv = 0.
		\end{equation*}
		
	\end{lem}

	\begin{proof}
		By Proposition \ref{prop regularity of Tijg for g in Lp} $Hv\in L^p(\R^3)$ for every $1<p<\infty$. Since by Lemma \ref{lemma Newtonian potential for smooth density}
		\begin{equation*}
		U^N_{v_j} \in C^3(\R^3), \quad j=1,2,3,
		\end{equation*}
		we have also
		\begin{equation*}
		Hv = \frac{1}{4\pi}\nabla \left( \sum_{j=1}^{3} \partial_{x_j} U^N_{v_j} \right) \in C^1(\R^3).
		\end{equation*}
		In particular $Hv\in L^p_{irr}(\R^3)$. Since $Hv$ is a gradient, its rotation is zero. For the divergence we have with Lemma \ref{lemma Newtonian potential for smooth density}
		\begin{equation*}
		\divergence Hv = \frac{1}{4\pi} \sum_{j=1}^{3}  \Delta U^N_{\partial_{y_j}v_j} = \sum_{j=1}^{3} \partial_{x_j}v_j = \divergence v.
		\end{equation*} 
		Further we get
		\begin{equation*}
		v_2 := v-Hv\in C^1 \cap L^p(\R^3)
		\end{equation*}
		for every $1<p<\infty$ with $\divergence ( v_2 ) = 0$ classically. Hence $v_2\in L^p_{sol}(\R^3)$.
	\end{proof}

	We are particularly interested in the Helmholtz-Weyl decomposition of vector fields $v\in L^p(\R^3)$. If $p<3$, we can show that $\partial_{x_j} U^N_{v_j}$ exists and, using the same formula as in Lemma \ref{lemma Helmholtz-Weyl decompostion for smooth v}, it is easy to deduct that in this situation, too, the Helmholtz-Weyl decomposition of $v$ is given by $Hv+(v-Hv)$. If $p\geq 3$ however (and this is the case of special interest in Mondian physics), the integral
	\begin{equation*}
	\partial_{x_j} U^N_{v_j} = \int\frac{x_j-y_j}{|x-y|^3} v_j(y) \diff y
	\end{equation*}
	does not necessarily converge. Nevertheless, also in this situation the Helmholtz-Weyl decomposition of $v$ is given by $Hv+(v-Hv)$ but we can no longer make use of the formula from Lemma \ref{lemma Helmholtz-Weyl decompostion for smooth v}. The key ingredients to prove this explicit form of the Helmholtz-Weyl decomposition for $v\in L^p(\R^3)$ is that $L^p_{irr}(\R^3)$ and $L^p_{sol}(\R^3)$ are closed subsets of $L^p(\R^3)$. For $L^p_{sol}(\R^3)$ this is proven in \cite{2011Galdi}. For $L^p_{irr}(\R^3)$ \citeauthor{2011Galdi} leaves this as an exercise to the reader (Exercise III.1.2). This exercise can be solved using Poincaré's inequality. This leads to the following theorem:
	
	\begin{thm}[Explicit Helmholtz-Weyl decomposition] \label{thm Helmholtz decomposition explicit}
		Let $1<p<\infty$ and $v\in L^p(\R^3)$ be a vector field. Then the uniquely determined Helmholtz-Weyl decomposition of $v$ is given by
		\begin{equation*}
		v = Hv + (v- Hv)
		\end{equation*}
		with
		\begin{equation*}
		Hv \in L^p_{irr}(\R^3) \quad \text{ and } \quad v - Hv \in L^p_{sol}(\R^3).
		\end{equation*}
	\end{thm}
	
	\begin{proof}
		We can approximate $v\in L^p(\R^3)$ with a sequence $(v_k)\subset C_c^2(\R^3)$. By Lemma \ref{lemma Helmholtz-Weyl decompostion for smooth v} the Helmholtz-Weyl decomposition of $v_k$ is given by
		\begin{equation*}
		v_k = Hv_k + (v_k-Hv_k)
		\end{equation*}
		with
		\begin{equation*}
		Hv_k \in L^p_{irr}(\R^3) \quad \text{ and } \quad v_k - Hv_k \in L^p_{sol}(\R^3).
		\end{equation*}
		By \cite[Exercise III.1.2]{2011Galdi} $L^p_{irr}(\R^3)$ is a closed subset of $L^p(\R^3)$. By \cite[Theorem III.2.3]{2011Galdi} $L^p_{sol}(\R^3)$ is the closure of the set
		\begin{equation*}
		\{v \in C_c^\infty(\R^3) \text{ with } \divergence v = 0 \}
		\end{equation*}
		with respect to the $L^p$-norm on $\R^3$. Hence $L^p_{sol}$ is a closed subset of $L^p(\R^3)$. Since  $v_k\rightarrow v$ and $Hv_k\rightarrow Hv$ in $L^p(\R^3)$ for $k\rightarrow\infty$
		\begin{equation*}
		Hv \in L^p_{irr}(\R^3) \quad \text{ and } \quad v - Hv \in L^p_{sol}(\R^3)
		\end{equation*}
		and
		\begin{equation*}
		v = Hv + (v-Hv)
		\end{equation*}
		is the uniquely determined Helmholtz-Weyl decomposition of $v$.
	\end{proof}
	
	Thus the vector field $Hv$ as defined in Definition \ref{definition Helmholtz operator H} is indeed the irrotational part of the vector field $v$ in the sense of the Helmholtz-Weyl decomposition. Before we close this section we prove two useful lemmas. First: For spherically symmetric vector fields the Helmholtz-Weyl decomposition is trivial.
	
	\begin{lem} \label{lemma H for v spherically symmetric}
		Let $1<p<\infty$. Then for every spherically symmetric vector field $v\in L^p(\R^3)$
		\begin{equation*}
		Hv = v.
		\end{equation*}
	\end{lem}
	
	\begin{proof}
		Let $v\in L^p(\R^3)$ be a spherically symmetric vector field. There exists a sequence $(v_k)\subset C_c^\infty(\R^3)$ of spherically symmetric vector fields with
		\begin{equation*}
		v_k \rightarrow v \quad \text{in } L^p(\R^3) \text{ for } k\rightarrow\infty.
		\end{equation*}
		Since the $v_k$ are spherically symmetric
		\begin{equation*}
		\rotation v_k = 0.
		\end{equation*}
		Hence, by standard results for vector calculus, there exist potentials $(U_k)\subset C^\infty(\R^3)$ such that for every $k\in\N$
		\begin{equation*}
		v_k = \nabla U_k,
		\end{equation*}
		in particular $v_k\in L^p_{irr}(\R^3)$ and the uniqueness of the Helmholtz-Weyl decomposition implies
		\begin{equation*}
		Hv_k = v_k.
		\end{equation*}
		Since $H:L^p(\R^3)\rightarrow L^p(\R^3)$ is continuous
		\begin{equation*}
		Hv = v.
		\end{equation*}
	\end{proof}
	
	And last in this section we prove the useful fact that the operator $H$ is symmetric.
	
	\begin{lem} \label{lemma int v cdot Hw = int Hv cdot w}
		Let $1<p,q<\infty$ with $\frac{1}{p} + \frac{1}{q} = 1$, and let $v\in L^p(\R^3)$, $w\in L^q(\R^3)$ be vector fields. Then
		\begin{equation*}
		\int v\cdot Hw \diff x = \int Hv \cdot w \diff x.
		\end{equation*}
	\end{lem}
	
	\begin{proof}
		Assume that $v\in L^1\cap L^p(\R^3)$ and $w\in L^1\cap L^q(\R^3)$. Since $v,w\in L^1(\R^3)$ we can apply Fubini and get that for every $\epsilon>0$ and $i,j=1,2,3$
		\begin{align*}
		\int T_{ij}^\epsilon v_j \, w_i \diff x & = - \iint _{|x-y|>\epsilon} \partial_{x_i} \partial_{x_j} \left( \frac{1}{|x-y|}\right) v_j(y)  \diff y \,w_i(x) \diff x \\
		& = - \int v_j(y) \int_{|x-y|>\epsilon} \partial_{y_i} \partial_{y_j} \left( \frac{1}{|x-y|}\right) w_i(x) \diff x \diff y \\
		& = \int v_j \, T_{ij}^\epsilon w_i \diff y.
		\end{align*}
		Hence by H\"older
		\begin{align*}
		\int Hv\cdot w \diff x & = \frac{1}{4\pi} \sum_{i,j=1}^3 \lim_{\epsilon\rightarrow 0} \int T_{ij}^\epsilon v_j \, w_i \diff x + \frac{1}{3} \sum_{i=1}^{3} \int v_iw_i \diff x \\
		& = \frac{1}{4\pi} \sum_{i,j=1}^3 \lim_{\epsilon\rightarrow 0} \int  v_j \, T_{ij}^\epsilon w_i \diff x + \frac{1}{3} \sum_{i=1}^{3} \int v_iw_i \diff x \\
		& = \int v \cdot Hw \diff x.
		\end{align*}
		Since $L^1 \cap L^p(\R^3)\subset L^p(\R^3)$ and $L^1 \cap L^q(\R^3)\subset L^q(\R^3)$ are dense, and $H$ is continuous, it follows that for every $v\in L^p(\R^3)$ and $w\in L^q(\R^3)$
		\begin{equation*}
		\int Hv \cdot w \diff x = \int v \cdot Hw \diff x.
		\end{equation*}
	\end{proof}

	\section{Mondian potentials} \label{section Mondian potentials}
	
	\cite{2010MNRAS.403..886Milgrom} introduced the QUMOND theory where the Mondian potential $U^M_\rho$ belonging to some density $\rho$ is given as the solution of the PDE \eqref{intro MOND field via QUMOND PDE}. \citeauthor{2010MNRAS.403..886Milgrom} gave an explicit formula for the solution of this PDE. Here in this paper we take another way to approach the Mondian potential $U^M_\rho$. This new approach enables us to place the entire QUMOND theory on a more robust, mathematical foundation. We define
	\begin{equation*}
		\nabla U^M_\rho \text{ is the irrotational part of }\nabla U^N_\rho + \lambda(|\nabla U^N_\rho|)\nabla U^N_\rho.
	\end{equation*}
	The theory of the previous section guarantees that the field $\nabla U^M_\rho$ is indeed the gradient of some potential $U^M_\rho$. Further, it guarantees that $\nabla U^M_\rho$ is an $L^p$ vector field for some $p>1$. To bring this new definition of $\nabla U^M_\rho$ together with the theory from \cite{2010MNRAS.403..886Milgrom}, we have to take a closer look on the potential $U^M_\rho$. How do we get an explicit formula for it? In Lemma \ref{lemma Helmholtz-Weyl decompostion for smooth v} we have seen that for a vector field
	\begin{equation*}
	v\in C_c^2(\R^3)
	\end{equation*}
	$Hv$ is the gradient of
	\begin{equation} \label{Mondian potentials introduction abstract formula for potential}
	\frac{1}{4\pi}\sum_{j=1}^3 \partial_{x_j} U^N_{v_j}.
	\end{equation}
	But if $\rho\in L^1\cap L^p(\R^3)$ for a $1<p<3$, then
	\begin{equation*}
	\nabla U^N_\rho \in L^q(\R^3)
	\end{equation*}
	for a $3/2<q<\infty$. If now $\lambda(\sigma) \approx \sqrt{a_0}/\sqrt{\sigma}$ then
	\begin{equation*}
	v := \lambda\left(|\nabla U^N_\rho|\right)\nabla U^N_\rho \in L^{2q}(\R^3)
	\end{equation*}
	with $2q > 3$. But then $\partial_{x_j} U^N_{v_j}$ is not well defined; for this it would be necessary that $v$ is an $L^p$ vector field for some $p<3$. However we can still recover a slight variation of Lemma \ref{lemma Helmholtz-Weyl decompostion for smooth v}. If we compare equation \eqref{Mondian potentials introduction abstract formula for potential} with the formula for the Mondian potential in \cite{2010MNRAS.403..886Milgrom}, we see that both are quite similar. We prove that the formula of \citeauthor{2010MNRAS.403..886Milgrom} really yields a well defined potential and that its gradient is given by $\nabla U^M_\rho$ as defined above.

	For better readability we decompse $U^M_\rho$ and write $U^M_\rho = U^N_\rho + U^\lambda_\rho$ where the gradient of $U^N_\rho$ is $\nabla U^N_\rho$ and the gradient of $U^\lambda_\rho$ shall be $H(\lambda(|\nabla U^N_\rho|)\nabla U^N_\rho)$. In the next lemma we analyse $U^\lambda_\rho$.
	
	\begin{lem} \label{lemma Ulambda}
		Assume that $\lambda:(0,\infty)\rightarrow(0,\infty)$ is measurable and that there is $\Lambda>0$ such that $\lambda(\sigma) \leq \Lambda/\sqrt{\sigma}$, for every $\sigma>0$ (thus the function $\lambda$ remains in its physically motivated regime). Let $\rho\in L^1\cap L^p(\R^3)$ for a $1<p<3$ and let $3/2<q<\infty$ with $2/3+1/p=1+1/q$. Set
		\begin{equation*}
		U^\lambda_\rho (x) := \frac{1}{4\pi} \int \lambda\left(|\nabla U^N_\rho(y)|\right)\nabla U^N_\rho(y) \cdot \left(\frac{x-y}{|x-y|^3} + \frac{y}{|y|^3}\right) \diff y, \quad x\in\R^3.
		\end{equation*}
		Then
		\begin{equation*}
		U^\lambda_\rho \in W^{1,2q}_{loc}(\R^3)
		\end{equation*}
		and
		\begin{equation*}
		\nabla U^\lambda_\rho = H\left(  \lambda\left(|\nabla U^N_\rho|\right)\nabla U^N_\rho \right) \in L^{2q}(\R^3).
		\end{equation*}
	\end{lem}

	\begin{rem*}
		The formula for $U^\lambda_\rho$ above is the one given by \cite{2010MNRAS.403..886Milgrom}. Observe its similarity with \eqref{Mondian potentials introduction abstract formula for potential}.
	\end{rem*}
	
	\begin{proof}
		Let $R>0$. First we prove that for
		\begin{equation*}
		I(x,y) := \lambda\left(|\nabla U^N_\rho(y)|\right)\nabla U^N_\rho(y) \cdot \left(\frac{x-y}{|x-y|^3} + \frac{y}{|y|^3}\right), \quad x,y\in\R^3,
		\end{equation*}
		holds
		\begin{equation} \label{equ proof Ulambda int |I| < infty}
		\iint_{|x|\leq R} |I(x,y)| \diff x\diff y < \infty.
		\end{equation}
		Let $p,q$ be as stated above and let $r$ be the dual exponent of $2q$. Since $3<2q<\infty$,
		\begin{equation*}
		1<r<\frac{3}{2}.
		\end{equation*}
		Then
		\begin{align*}
		\iint_{|x|\leq R, |y|\leq 2R} |I(x,y)| \diff x \diff y 
		&\leq \Lambda \iint_{|x|\leq R, |y|\leq 2R} |\nabla U^N_\rho(y)|^{1/2} \left( \frac{1}{|x-y|^2} + \frac{1}{|y|^2}\right) \diff x \diff y \\
		& \leq 2\Lambda\measure(B_R) \|\nabla U^N_\rho\|_q^{1/2} \left\| \frac{1}{|y|^2}\right\|_{L^r(B_{3R})} < \infty.
		\end{align*}
		Next observe that for all $y\in \R^3\backslash\{0\}$, $i,j=1,2,3$
		\begin{equation*}
		\left| \partial_{y_i} \frac{y_j}{|y|^3} \right| = \left| \frac{\delta_{ij}}{|y|^3} - 3\frac{y_iy_j}{|y|^5} \right| \leq \frac{4}{|y|^3}.
		\end{equation*}
		Thus for $x,y\in\R^3$ with $|x|\leq R$, $|y|> 2R$ holds
		\begin{equation*}
		\left| \frac{x_j-y_j}{|x-y|^3} - \frac{y_j}{|y|^3} \right| = \left| \int_0^1 \frac{\diff}{\diff s} \frac{y_j-sx_j}{|y-sx|^3} \diff s\right| \leq R\int_0^1 \frac{\diff s}{|y-sx|^3}.
		\end{equation*}
		Since for all $0\leq s \leq 1$
		\begin{equation*}
		|y-sx| \geq \frac{|y|}{2},
		\end{equation*}
		we estimate further
		\begin{equation*}
		\left| \frac{x_j-y_j}{|x-y|^3} - \frac{y_j}{|y|^3} \right| \leq \frac{8R}{|y|^3}.
		\end{equation*}
		Hence
		\begin{align*}
		\iint_{|x|\leq R, |y|> 2R} |I(x,y)| \diff x \diff y 
		&\leq C \iint_{|x|\leq R, |y|> 2R} |\nabla U^N_\rho(y)|^{1/2} \frac{1}{|y|^3} \diff y \\
		& \leq C\measure(B_R) \|\nabla U^N_\rho \|_q^{1/2} \left\|\frac{1}{|y|^3}\right\|_{L^r(\{|y|>2R\})} \\&< \infty.
		\end{align*}
		Thus \eqref{equ proof Ulambda int |I| < infty} holds. Fubini then implies that
		\begin{equation*}
		U^\lambda_\rho \in L^1_{loc}(\R^3).
		\end{equation*}
		Hence the potential $U^\lambda_\rho$ as given by the formula from \cite{2010MNRAS.403..886Milgrom} is well defined. It remains to prove that
		\begin{equation*}
		\nabla U^\lambda_\rho = H\left(  \lambda\left(|\nabla U^N_\rho|\right)\nabla U^N_\rho \right).
		\end{equation*}
		We write shortly
		\begin{equation*}
		v := \frac{1}{4\pi}\lambda\left(|\nabla U^N_\rho|\right)\nabla U^N_\rho \in L^{2q}(\R^3).
		\end{equation*}
		Let $\phi\in C_c^\infty(\R^3)$. Then
		\begin{equation*}
		\int U^\lambda_\rho \partial_{x_i}\phi \diff x = \iint v(y)\cdot \left( \frac{x-y}{|x-y|^3} + \frac{y}{|y|^3} \right) \partial_{x_i}\phi(x) \diff y \diff x.
		\end{equation*}
		Thanks to \eqref{equ proof Ulambda int |I| < infty} we can apply Fubini and, since
		\begin{equation*}
		\int\partial_{x_i}\phi(x) \diff x = 0,
		\end{equation*}
		we have
		\begin{align*}
		\int U^\lambda_\rho\partial_{x_i}\phi\diff x & = - \int v\cdot \nabla U^N_{\partial_{x_i}\phi} \diff y = - \sum_{j=1}^3 \int v_j \partial_{y_i}\partial_{y_j} U^N_\phi \diff y.
		\end{align*}
		Lemma \ref{lemma Newtonian potential for smooth density} and Proposition \ref{prop regularity of Tijg for g in Lp} imply
		\begin{equation*}
		\int U^\lambda_\rho\partial_{x_i}\phi\diff x = -\sum_{j=1}^3\lim_{\epsilon\rightarrow 0} \int v_j \left( T_{ij}^\epsilon\phi + \delta_{ij}\frac{4\pi}{3}\phi \right) \diff y.
		\end{equation*}
		As in the proof of Lemma \ref{lemma int v cdot Hw = int Hv cdot w} we have
		\begin{equation*}
		\int v_j\, T_{ij}^\epsilon \phi \diff y = \int T_{ij}^\epsilon v_j \,\phi\diff y.
		\end{equation*}
		Hence
		\begin{align*}
		\int U^\lambda_\rho\partial_{x_i}\phi\diff x & = - 4\pi \int \left( \frac{1}{4\pi} \sum_{j=1}^3 T_{ij} v_j + \frac{1}{3}v_i \right) \phi \diff y \\
		& = -4\pi \int H_i v\, \phi \diff y.
		\end{align*}
		Thus
		\begin{equation*}
		\nabla U^\lambda_\rho = 4\pi Hv = H\left( \lambda\left(|\nabla U^N_\rho|\right)\nabla U^N_\rho \right).
		\end{equation*}
		In particular, the Helmholtz-Weyl decomposition Theorem \ref{thm Helmholtz decomposition explicit} implies
		\begin{equation*}
		U^\lambda_\rho \in W^{1,2q}_{loc}(\R^3)
		\end{equation*}
		and
		\begin{equation*}
			\nabla U^\lambda_\rho \in L^{2q}(\R^3).
		\end{equation*}

	\end{proof}

	Taking Lemma \ref{lemma Ulambda} and \ref{lemma Newtonian potential for density in Lp NEW} together gives
	
	\begin{cor}
		Let $\lambda$ be as in Lemma \ref{lemma Ulambda} and $\rho \in L^1\cap L^p(\R^3)$ for some $p>1$. Then
		\begin{equation*}
			U^M_\rho := U^N_\rho + U^\lambda_\rho
		\end{equation*}
		is well defined and once weakly differentiable. The gradient $\nabla U^M_\rho$ can be decomposed into an $L^q$ plus an $L^r$ vector field, where $\nabla U^N_\rho$ is an $L^q$ vector field for some $q>3/2$ and $\nabla U^\lambda_\rho$ is an $L^r$ vector field for some $r>3$.
		
	\end{cor}

	\begin{proof}
		When $\rho \in L^1 \cap L^p(\R^3)$ it follows from the interpolation formula that $\rho\in L^{p'}$ for every $1 < p' < p$, in particular for $p'< 3$. Thus Lemma \ref{lemma Newtonian potential for density in Lp NEW} and \ref{lemma Ulambda} imply that $\nabla U^N_\rho$ and $\nabla U^\lambda_\rho$ are both well defined and that $\nabla U^N_\rho\in L^q(\R^3)$ for some $q>3/2$ and $\nabla U^\lambda_\rho\in L^r(\R^3)$ for some $r>3$.
	\end{proof}
	
	To summarize, we have now proven that the potential $U^M_\rho$ as defined by \cite{2010MNRAS.403..886Milgrom} is really well defined. It is once weakly differentiable and its gradient $\nabla U^M_\rho$ is the irrotational part of the vector field $\nabla U^N_\rho + \lambda(|\nabla U^N_\rho|)\nabla U^N_\rho$ in the sense of the Helmholtz-Weyl decomposition. Further, when we write $\nabla U^M_\rho = \nabla U^N_\rho + \nabla U^\lambda_\rho$, we have that $\nabla U^N_\rho$ is an $L^p$ vector field for some $p>3/2$ and that $\nabla U^\lambda_\rho$ is an $L^p$ vector field for some $p>3$. Lastly, from the definition of the operator $H$ (Definition \ref{definition Helmholtz operator H}) we have a new explicit formula for $\nabla U^\lambda_\rho$ using singular integral operators. And this new formula is not only useful to analyse the regularity of $\nabla U^\lambda_\rho$ as done above but it is also very useful to verify that $U^M_\rho$ is really a solution of the PDE \eqref{intro MOND field via QUMOND PDE} from the introduction. This we do in the last lemma of this section.
	
	\begin{lem}
		Let $\lambda$ be as in Lemma \ref{lemma Ulambda} and $\rho\in L^1\cap L^p(\R^3)$ for some $p>1$, then $U^M_\rho$ solves the PDE
		\begin{equation*}
			\divergence\left( \nabla U^M_\rho\right) = \divergence \left( \nabla U^N_\rho + \lambda\left(\left| \nabla U^N_\rho \right|\right) \nabla U^N_\rho \right)
		\end{equation*}
		in distribution sense, i.e., for every $\phi \in C_c^\infty(\R^3)$
		\begin{equation*}
			\int \nabla U^M_\rho \cdot \nabla \phi \diff x = \int \left( \nabla U^N_\rho + \lambda\left(\left| \nabla U^N_\rho \right|\right) \nabla U^N_\rho \right) \cdot \nabla \phi \diff x.
		\end{equation*}
	\end{lem}

	\begin{proof}
		We have
		\begin{equation*}
			\nabla U^M_\rho = H \left(\nabla U^N_\rho +  \lambda\left(\left| \nabla U^N_\rho \right|\right) \nabla U^N_\rho \right).
		\end{equation*}
		Since $\nabla U^N_\rho$ is already a gradient, we get $H(\nabla U^N_\rho) = \nabla U^N_\rho$. Thus
		\begin{align*}
			\int \nabla U^M_\rho \cdot \nabla \phi \diff x 
			& = \int  H \left(\nabla U^N_\rho +  \lambda\left(\left| \nabla U^N_\rho \right|\right) \nabla U^N_\rho \right)  \cdot \nabla \phi \diff x			\\
			& = \int \left[ \nabla U^N_\rho + H \left( \lambda\left(\left| \nabla U^N_\rho \right|\right) \nabla U^N_\rho \right) \right] \cdot \nabla \phi \diff x.
		\end{align*}
		The operator $H$ is symmetric (Lemma \ref{lemma int v cdot Hw = int Hv cdot w}) and hence
		\begin{align*}
		\int \nabla U^M_\rho \cdot \nabla \phi \diff x 
			= \int \nabla U^N_\rho \cdot \nabla \phi \diff x + \int \lambda\left(\left| \nabla U^N_\rho \right|\right) \nabla U^N_\rho \cdot H(\nabla \phi) \diff x.
		\end{align*}
		$H(\nabla \phi) = \nabla \phi$ since $\nabla \phi$ is already a gradient. Thus it follows that
		\begin{align*}
			\int \nabla U^M_\rho \cdot \nabla \phi \diff x 
				& = \int \nabla U^N_\rho \cdot \nabla \phi \diff x + \int \lambda\left(\left| \nabla U^N_\rho \right|\right) \nabla U^N_\rho \cdot \nabla \phi \diff x \\
				& = \int \left( \nabla U^N_\rho + \lambda\left(\left| \nabla U^N_\rho \right|\right) \nabla U^N_\rho \right) \cdot \nabla \phi \diff x.
		\end{align*}
		This means that $U^M_\rho$ solves the PDE
		\begin{equation*}
		\divergence\left( \nabla U^M_\rho\right) = \divergence \left( \nabla U^N_\rho + \lambda\left(\left| \nabla U^N_\rho \right|\right) \nabla U^N_\rho \right)
		\end{equation*}
		in distribution sense.
		
	\end{proof}

	Taking all statements of the lemmas proven in this section together implies that Theorem \ref{thm main} from the introduction holds.

	\section{Second derivatives of Mondian potentials} \label{section second derivatives}
	
	In the previous section we have shown that in QUMOND the potential $U^M_\rho$, which corresponds to some density $\rho$ on $\R^3$, is well defined, once weakly differentiable and solves the second order PDE \eqref{intro MOND field via QUMOND PDE} in distribution sense. Does $U^M_\rho$ also have second order derivatives?
	
	Consider a density $\rho \in L^1\cap L^\infty(\R^3)$ with finite support. Such a density is a reasonable model for the distribution of mass in systems like globular clusters or galaxies. If we have such a density the interpolation formula and Lemma \ref{lemma Newtonian potential for density in Lp NEW} tell us that the second derivatives of the corresponding Newtonian potential are $L^p$-functions:
	\begin{equation*}
		D^2 U^N_\rho \in L^p(\R^3)
	\end{equation*}
	for every $p\in (1,\infty)$. Therefore
	\begin{equation*}
		\nabla U^N_\rho \in C^{0,\alpha}(\R^3) 
	\end{equation*}
	for every $\alpha \in (0,1)$ by Morrey's inequality; the space $C^{0,\alpha}(\R^3)$ denotes the space of bounded and H\"older continuous functions with H\"older exponent $\alpha$ on $\R^3$. Under quite general assumptions (see Lemma \ref{lemma lambda is Hoelder continuous}) the function
	\begin{equation*}
		\R^3 \ni u \mapsto \lambda(|u|)u \in \R^3
	\end{equation*} 
	is H\"older continuous with H\"older exponent $\frac 1 2$. Thus
	\begin{equation*}
		\lambda(|\nabla U^N_\rho|) \nabla U^N_\rho \in C^{0,\beta}(\R^3)
	\end{equation*}
	for every $\beta \in (0,\frac 1 2)$. Assuming for simplicity that we are in spherical symmetry, we have by Lemma \ref{lemma H for v spherically symmetric}
	\begin{equation*}
		\nabla U^\lambda_\rho = \lambda(|\nabla U^N_\rho|) \nabla U^N_\rho \in C^{0,\beta}(\R^3).
	\end{equation*}
	Taking a second look on Morrey's inequality one could now expect that
	\begin{equation*}
		D^2 U^\lambda_\rho \in L^p(\R^3)
	\end{equation*}
	for all $1<p<6$. But this expectation proves deceptive. Why? Let us remain in the situation of spherical symmetry and for $\rho=\rho(x)$ spherically symmetric study the divergence of
	\begin{equation*}
	\nabla U^M_\rho = \nabla U^N_\rho + \lambda(|\nabla U^N_\rho|)\nabla U^N_\rho.
	\end{equation*}
	In view of Lemma \ref{lemma Newtons shell theorem}
	\begin{equation*}
	\lambda(|\nabla U^N_\rho(x)|)\nabla U^N_\rho(x) = \frac{\sqrt{M(r)}}{r}\frac{x}{r}, \quad r = |x|,
	\end{equation*}
	where for convenience we assumed $\lambda(\sigma)=1/\sqrt \sigma$, $\sigma>0$. So
	\begin{align*}
	\divergence(\nabla U^M_\rho(x)) & = \Delta U^N_\rho(x) + \frac{1}{r^2}(r\sqrt{M(r)})' \\
	& = 4\pi\rho(r) + \frac{\sqrt{M(r)}}{r^2} + \frac{\sqrt{M(r)}'}{r}.
	\end{align*}
	$\rho$ is just fine, the second term can be controlled as expected above, but the third one will cause problems. We prove the following
	
	\begin{prop} \label{prop bounds on ( sqrt(M)/r ) Prime}
		Let $R>0$, $1<p,q<\infty$ and $\rho\in C^1\cap L^p(\R^3)$, $\geq 0$, spherically symmetric. Then
		\begin{equation*}
		\left\| \frac{\sqrt{M(r)}}{r^2} \right\|_{L^q(B_R)} \leq C \|\rho\|_p^{1/2}
		\end{equation*}
		if $1< q < 6$ and $p > 3q/(6-q)$ with $C=C(p,q,R)>0$, and
		\begin{equation*}
		\left\|\frac{\sqrt{M(r)}'}{r}\right\|_{L^q(B_R)} \leq C \|\rho\|_p^{1/2}
		\end{equation*}
		if $1<q<2$ and $p > q/(2-q) + q$ with $C=C(p,q,R)>0$. With $\sqrt{M(r)}'$ we denote the function
		\begin{equation*}
		\sqrt{M(r)}' := \begin{cases}
		\frac{2\pi r^2\rho(r)}{\sqrt{M(r)}} & \text{, if } M(r)>0 \\
		0 & \text{, if } M(r) = 0
		\end{cases}.
		\end{equation*}
	\end{prop}
	
	\begin{proof}
		Let $1< q < 6$ and $1<p<\infty$ with $p > 3q/(6-q)$. For $r\geq 0$
		\begin{equation*}
		M(r) = \int_{B_r}\rho(x)\diff x \leq \|\rho\|_p\|1_{B_r}\|_{p/(p-1)} \leq C\|\rho\|_p r^{3-3/p}.
		\end{equation*}
		Thus
		\begin{equation*}
		\left\| \frac{\sqrt{M(r)}}{r^2} \right\|_{L^q(B_R)}^q = \int_{B_R} M(r)^{q/2}r^{-2q} \diff x \leq C\|\rho\|_p^{q/2} \int_{B_R} r^{\frac{3q}{2}-\frac{3q}{2p}-2q} \diff x.
		\end{equation*}
		Since
		\begin{equation*}
		\frac{3q}{2}-\frac{3q}{2p}-2q > -3 \Leftrightarrow 3-\frac q2 > \frac{3q}{2p} \Leftrightarrow p > \frac{3q}{6-q}
		\end{equation*}
		we have
		\begin{equation*}
		\left\| \frac{\sqrt{M(r)}}{r^2} \right\|_{L^q(B_R)} \leq C \|\rho\|_p^{1/2}.
		\end{equation*}
		
		Now we turn to the second estimate. Let $1<q<2$, $p > q/(2-q) + q$ and $r_0\geq 0$ be such that $M(r_0) = 0$ and $M(r)>0$ for all $r>r_0$. Since $\rho\in C^1(\R^3)$, $M(r)\in C^1([0,\infty))$ with
		\begin{equation*}
		M'(r) = \frac{\diff}{\diff r} 4\pi \int_0^r s^2\rho(s) \diff s = 4\pi r^2\rho(r).
		\end{equation*}
		Hence $\sqrt{(M(r))} \in C^1((r_0,\infty))$ with
		\begin{equation*}
		\sqrt{M(r)}' = \frac{2\pi r^2\rho(r)}{\sqrt{M(r)}}, \quad r>r_0.
		\end{equation*}
		Assume that $R>r_0$, then
		\begin{align*}
		\left\| \frac{\sqrt{M(r)}'}{r} \right\|^q_{L^q(B_R\backslash B_{r_0})} & = \int_{B_R\backslash B_{r_0}} \left( \frac{2\pi r\rho(r)}{\sqrt{M(r)}} \right)^q \diff x \\
		& \leq (2\pi R)^q \int_{B_R\backslash B_{r_0}} \rho(r)^\alpha \frac{\rho(r)^{q-\alpha}}{M(r)^{q/2}} \diff x
		\end{align*}
		with
		\begin{equation*}
		\alpha := \frac{p}{p-1}(q-1).
		\end{equation*}
		Obviously $\alpha>0$, and further $\alpha < q$ since
		\begin{equation*}
		\alpha = \frac{p}{p-1}(q-1) < q \Leftrightarrow 1 - \frac{1}{q} < 1 - \frac{1}{p} \Leftrightarrow q < p.
		\end{equation*}
		Now we apply H\"older's inequality and get
		\begin{equation*}
		\left\| \frac{\sqrt{M(r)}'}{r} \right\|^q_{L^q(B_R\backslash B_{r_0})} \leq C \|\rho\|_p^\alpha \left\| \frac{\rho(r)^{q-\alpha}}{M(r)^{q/2}} \right\|_{L^{p/(p-\alpha)}(B_R\backslash B_{r_0})};
		\end{equation*}
		note that $0<\alpha < q < p$. Since
		\begin{equation*}
		(q-\alpha)\frac{p}{p-\alpha} = 1 \Leftrightarrow q - \alpha = 1 - \frac \alpha p \Leftrightarrow q - 1 = \alpha \left( 1 - \frac 1p \right) \Leftrightarrow \alpha = \frac{p}{p-1}(q-1),
		\end{equation*}
		we have
		\begin{align*}
		\left\| \frac{\rho(r)^{q-\alpha}}{M(r)^{q/2}} \right\|_{L^{p/(p-\alpha)}(B_R\backslash B_{r_0})} & = \left( \int_{B_R\backslash B_{r_0}} \rho(r)M(r)^{-pq/(2p-2\alpha)} \diff x \right)^{(p-\alpha)/p} \\
		& = C \left[ \int_{r_0}^R  \left( M(r)^{1-pq/(2p-2\alpha)} \right)' \diff r \right]^{(p-\alpha)/p};
		\end{align*}
		here we have used that $pq/(2p-2\alpha)<1$ since
		\begin{align*}
		\frac{pq}{2p-2\alpha} < 1 & \Leftrightarrow \frac{q}{2} < 1 -\frac{\alpha}{p} = 1 - \frac{q-1}{p-1} \\
		& \Leftrightarrow \frac{2-q}{2} > \frac{q-1}{p-1} \\
		& \Leftrightarrow p > 1 + \frac{2(q-1)}{2-q} = \frac{2 - q + 2q -2}{2-q} = \frac{q}{2-q}.
		\end{align*}
		Thus
		\begin{equation*}
		\left\| \frac{\sqrt{M(r)}'}{r} \right\|^q_{L^q(B_R\backslash B_{r_0})} \leq C \|\rho\|_p^\alpha \|\rho\|_{L^1(B_R)}^{(p-\alpha)/p - q/2}.
		\end{equation*}
		Since
		\begin{equation*}
		\|\rho\|_{L^1(B_R)} \leq C\|\rho\|_p
		\end{equation*}
		and
		\begin{align*}
		\frac 1q \left(\alpha + \frac{p-\alpha}{p} - \frac q2\right) & = \frac{q-1}{q}\frac{p}{p-1} + \frac 1q \left( 1 - \frac{q-1}{p-1} \right) - \frac 12 \\
		& = \frac{(q-1)p+(p-q)}{q(p-1)} - \frac{1}{2} \\
		& = \frac{pq - q}{q(p-1)} - \frac 12 \\
		&= \frac 12,
		\end{align*}
		we finally have
		\begin{equation*}
		\left\| \frac{\sqrt{M(r)}'}{r} \right\|_{L^q(B_R\backslash B_{r_0})} \leq C\|\rho\|_q^{1/2}.
		\end{equation*}
		
	\end{proof}
	
 	Now we can analyse derivatives of $\lambda(|\nabla U^N_\rho|)\nabla U^N_\rho$. However, before we do so, we need to strengthen the assumptions on $\lambda$ that we made in Lemma \ref{lemma Ulambda}. This we do in the next lemma, where we take a look on the derivative $\lambda'$. The assumptions of the next lemma imply that $\lambda$ has the same regularity as in Lemma \ref{lemma Ulambda} and that additionally the function $\R^3\ni u\mapsto \lambda(|u|)u$ is H\"older continuous.
	
	\begin{lem} \label{lemma lambda is Hoelder continuous}
		Assume that $\lambda\in C^1((0,\infty))$, $\lambda(\sigma)\rightarrow0$ as $\sigma\rightarrow\infty$ and there is $\Lambda>0$ such that $-\Lambda/(2\sigma^{3/2}) \leq \lambda'(\sigma)\leq 0$, for $\sigma>0$. Then
		\begin{equation*}
			\lambda(\sigma) \leq \Lambda/\sqrt{\sigma}
		\end{equation*}
		for every $\sigma>0$ (as in Lemma \ref{lemma Ulambda})
		and there is a $C>0$ such that for all $u,v\in\R^3$
		\begin{equation*}
		|\lambda(|u|)u - \lambda(|v|)v| \leq C|u-v|^{1/2}
		\end{equation*}
		with $\lambda(|u|)u = 0$ if $u=0$.
	\end{lem}

	We postpone the proof of this Lemma to the appendix and return our attention to the analysis of the second derivatives of $U^M_\rho$. Using Proposition \ref{prop bounds on ( sqrt(M)/r ) Prime} we can control $L^q$-norms of the derivatives of the Mondian part
	\begin{equation*}
	\lambda(|\nabla U^N_\rho|)\nabla U^N_\rho
	\end{equation*}
	of the field $\nabla U^M_\rho$ provided $1<q<2$ and $\rho\geq 0$ is spherically symmetric.
	
	\begin{lem} \label{lemma bound for derivative of lambda(nabla UM)nabla UM}
		Let $1<q<2$, $p>\frac{q}{2-q} + q$, $R>0$ and $\rho\in L^1\cap L^p(\R^3)$, $\geq 0$, spherically symmetric. Assume that $\lambda$ is as in Lemma \ref{lemma lambda is Hoelder continuous}, then
		\begin{equation*}
		\lambda(|\nabla U^N_\rho|)\nabla U^N_\rho \in W^{1,q}_{loc}(\R^3)
		\end{equation*}
		with
		\begin{equation*}
		\left\| \nabla \left[\lambda(|\nabla U^N_\rho|)\nabla U^N_\rho \right] \right\|_{L^q(B_R)} \leq C\|\rho\|_p^{1/2}
		\end{equation*}
		where $C=C(p,q,R)>0$.
	\end{lem}

	\begin{proof}
		Since we are in spherical symmetry, Lemma \ref{lemma Newtons shell theorem} gives
		\begin{equation*}
		\lambda(|\nabla U^N_\rho|)\nabla U^N_\rho = \lambda\left( \frac{M(r)}{r^2} \right) \frac{M(r)}{r^2} \frac xr, \quad x\in\R^3, \, r=|x|;
		\end{equation*}
		for better readability we suppress the $x$-argument on the left side.
		Using the abbreviation
		\begin{equation*}
		\tilde\lambda(\sigma) = \lambda(\sigma)\sigma, \quad \sigma\geq 0,
		\end{equation*}
		we have
		\begin{equation*}
		\lambda(|\nabla U^N_\rho|)\nabla U^N_\rho = \tilde\lambda\left( \frac{M(r)}{r^2} \right)\frac xr.
		\end{equation*}
		Thanks to Lemma \ref{lemma lambda is Hoelder continuous}
		\begin{equation} \label{tilde lambda Hoelder continuous}
		|\tilde\lambda(\sigma) - \tilde\lambda(\tau) | \leq C|\sigma-\tau|^{1/2}, \quad \sigma,\tau\geq 0,
		\end{equation}
		for a $C>0$ where
		\begin{equation*}
		\tilde\lambda(0) = 0.
		\end{equation*}
		From this lemma follows further
		\begin{equation} \label{tilde lambda bounded from above}
		0 \leq \tilde\lambda(\sigma) \leq \Lambda\sqrt \sigma, \quad \sigma\geq 0.
		\end{equation}
		Using the bounds for $\lambda$ and $\lambda'$ we get
		\begin{equation} \label{tilde lambda Prime bounded}
		|\tilde\lambda'(\sigma) | \leq |\lambda'(\sigma)|\sigma + \lambda(\sigma) \leq \frac{C}{\sqrt \sigma}, \quad \sigma>0,
		\end{equation}
		for a $C>0$. Thanks to \eqref{tilde lambda bounded from above}, for every $R>0$ holds
		\begin{equation*}
		\left\| \lambda(|\nabla U^N_\rho|)\nabla U^N_\rho \right\|_{L^q(B_R)}^q \leq \Lambda^q \int_{B_R} \left( \frac{\sqrt{M(r)}}{r}\right)^q \diff x \leq C\|\rho\|_1^{q/2}.
		\end{equation*}
		
		Next we approximate $\rho$ by smooth densities $\rho_n$ and study the (weak) derivatives of $\lambda(|\nabla U^N_{\rho_n}|)\nabla U^N_{\rho_n}$. Let $(\rho_n)\subset C_c^1(\R^3)$ be a sequence of spherically symmetric densities such that
		\begin{equation*}
		\rho_n \rightarrow \rho \quad \text{strongly in } L^1(\R^3) \text{ and } L^p(\R^3) \text{ for } n\rightarrow \infty.
		\end{equation*}
		As above $\lambda(|\nabla U^N_{\rho_n}|)\nabla U^N_{\rho_n} \in L^q_{loc}(\R^3)$. Denote by 
		\begin{equation*}
		M_n(r) = \int_{B_r} \rho_n \diff x, \quad r\geq 0,
		\end{equation*}
		the mass of $\rho_n$ inside the ball with radius $r$. Then $M_n\in C^1(\R^3)$ with
		\begin{equation*}
		\nabla (M_n(r)) =  M'_n(r)\frac{x}{r} = 4\pi\rho_n(r)rx.
		\end{equation*}
		Let $r_n\geq 0$ be such that $M_n(r_n) = 0$ and $M_n(r)>0$ for all $r>r_n$. Then
		\begin{equation*}
		\lambda(|\nabla U^N_{\rho_n}|)\nabla U^N_{\rho_n} \in C^1(\R^3\backslash\{|x|=r_n\})
		\end{equation*}
		with
		\begin{equation} \label{equ derivative of MOND correction central part}
		\partial_{x_i}\left[\lambda(|\nabla U^N_{\rho_n}|) \partial_{x_j} U^N_{\rho_n}\right]  = 0
		\end{equation}
		if $|x|<r_n$, and
		\begin{align} \label{equ derivative of MOND correction outer part}
		\partial_{x_i}\left[\lambda(|\nabla U^N_{\rho_n}|) \partial_{x_j} U^N_{\rho_n}\right] = & \, \partial_{x_i} \left( \tilde\lambda \left(\frac{M_n(r)}{r^2}\right) \frac{x_j}{r} \right) \nonumber\\
		= & \, \tilde\lambda'\left( \frac{M_n(r)}{r^2} \right) M_n'(r) \frac{x_ix_j}{r^4} \\
		& \,- 2 \tilde\lambda'\left( \frac{M_n(r)}{r^2} \right) M_n(r) \frac{x_ix_j}{r^5} \nonumber\\
		& \, + \tilde\lambda\left( \frac{M_n(r)}{r^2} \right) \left( \frac{\delta_{ij}}{r} - \frac{x_ix_j}{r^3} \right) \nonumber
		\end{align}
		if $|x|>r_n$ and $i,j=1,2,3$. Denote by
		\begin{equation*}
		\partial_{x_i}\left[\lambda(|\nabla U^N_{\rho_n}|) \partial_{x_j} U^N_{\rho_n}\right]
		\end{equation*}
		the functions that are pointwise a.e. defined by \eqref{equ derivative of MOND correction central part} and \eqref{equ derivative of MOND correction outer part}. Using \eqref{tilde lambda bounded from above} and \eqref{tilde lambda Prime bounded} we get for $|x|>r_n$ 
		\begin{align*}
		\left| \partial_{x_i}\left[\lambda(|\nabla U^N_{\rho_n}|) \partial_{x_j} U^N_{\rho_n}\right] \right| & \leq  C\left( \frac{M'_n(r)}{2\sqrt{M_n(r)}} \frac 1r + \frac{\sqrt{M_n(r)}}{r^2} \right) \\
		& = C \left( \frac{\sqrt{M_n(r)}'}{r} + \frac{\sqrt{M_n(r)}}{r^2} \right).
		\end{align*}
		Since $p > q/(2-q) + q$ and
		\begin{equation*}
		\frac{q}{2-q} = \frac{3q}{6-3q} > \frac{3q}{6-q},
		\end{equation*}
		we can apply Proposition \ref{prop bounds on ( sqrt(M)/r ) Prime} and get for every $R>0$
		\begin{equation} \label{equ derivative of MOND correction estimate}
		\left\| \partial_{x_i}\left[\lambda(|\nabla U^N_{\rho_n}|) \partial_{x_j} U^N_{\rho_n}\right] \right\|_{L^q(B_R)} \leq C\|\rho_n\|_p^{1/2}.
		\end{equation}
		
		Now we prove that the functions given by \eqref{equ derivative of MOND correction central part} and \eqref{equ derivative of MOND correction outer part}
		are indeed the weak derivatives of $\lambda(|\nabla U^N_{\rho_n}|) \nabla U^N_{\rho_n}$. For every $\phi\in C_c^\infty(\R^3)$
		\begin{align*}
		\int \lambda(|\nabla U^N_{\rho_n}|) \partial_{x_j} U^N_{\rho_n} \partial_{x_i}\phi \diff x = & \, \lim_{s \searrow r_n} \int_{\{|x|\geq s\}}  \lambda(|\nabla U^N_{\rho_n}|) \partial_{x_j} U^N_{\rho_n} \partial_{x_i}\phi \diff x \\
		= & \, - \int \partial_{x_i}\left(\lambda(|\nabla U^N_{\rho_n}|) \partial_{x_j} U^N_{\rho_n}\right) \phi \diff x \\
		& \, +	\lim_{s \searrow r_n} \int_{\{|x|=s\}}  \lambda(|\nabla U^N_{\rho_n}|) \partial_{x_j} U^N_{\rho_n} \phi \frac{x_i}{|x|} \diffS (x).
		\end{align*}
		If $r_n=0$, we use
		\begin{equation*}
		\left|\lambda(|\nabla U^N_{\rho_n}|) \nabla U^N_{\rho_n}\right| \leq \frac{\Lambda\sqrt{M_n(r)}}{r} \leq \frac{\Lambda\|\rho\|_1^{1/2}}{r}
		\end{equation*}
		and get
		\begin{equation*}
		\left| \int_{\{|x|=s\}}  \lambda(|\nabla U^N_{\rho_n}|) \partial_{x_j} U^N_{\rho_n} \phi \frac{x_i}{|x|} \diffS (x) \right| \leq Cs \rightarrow 0 \quad\text{for } s\rightarrow 0.
		\end{equation*}
		If $r_n>0$, we use
		\begin{equation*}
		\left|\lambda(|\nabla U^N_{\rho_n}|) \nabla U^N_{\rho_n}\right| \leq \frac{\Lambda}{r_n} \left( \int_{r_n<|x|<s} \rho_n \diff x \right)^{1/2} \rightarrow 0 \quad \text{for } s \rightarrow r_n,
		\end{equation*}
		and get, too, that the border term in the above integration by parts vanishes. Hence the by \eqref{equ derivative of MOND correction central part} and \eqref{equ derivative of MOND correction outer part} pointwise a.e. defined functions are indeed the weak derivatives of
		\begin{equation*}
		\lambda(|\nabla U^N_{\rho_n}|) \nabla U^N_{\rho_n} \in W^{1,q}_{loc}(\R^3).
		\end{equation*}
		
		It remains to prove that
		\begin{equation*}
		\lambda(|\nabla U^N_{\rho}|) \partial_{x_j} U^N_{\rho} \in W^{1,q}_{loc}(\R^3)
		\end{equation*}
		and that the estimate \eqref{equ derivative of MOND correction estimate} holds with $\rho_n$ replaced by $\rho$. Using \eqref{tilde lambda Hoelder continuous} and H\"older we have for $R>0$
		\begin{align*}
		\left\| \lambda(|\nabla U^N_{\rho_n}|) \nabla U^N_{\rho_n} - \lambda(|\nabla  U^N_{\rho}|) \partial_{x_j} U^N_{\rho} \right\|_{L^1(B_R)} &=\int_{B_R} \left| \tilde\lambda\left( \frac{M(r)}{r^2} \right) - \tilde\lambda\left( \frac{M_n(r)}{r^2} \right) \right| \diff x \\
		& \leq C \int_{B_R} \frac{|M(r)-M_n(r)|^{1/2}}{r} \diff x \\
		& \leq C  \left( \int_{B_R} |M_n(r) - M(r)| \diff x \right)^{1/2} \\
		& \leq C \| \rho_n-\rho\|_1^{1/2}.
		\end{align*}
		Thus
		\begin{equation*}
		\lambda(|\nabla U^N_{\rho_n}|) \nabla U^N_{\rho_n} \rightarrow \lambda(|\nabla U^N_{\rho}|) \partial_{x_j} U^N_{\rho} \quad \text{strongly in }L^1(B_R) \text{ for }n\rightarrow\infty.
		\end{equation*}
		Since
		\begin{equation*}
		\|\rho_n\|_p \leq C
		\end{equation*}
		independent of $n\in\N$, \eqref{equ derivative of MOND correction estimate} implies that there is a subsequence (again denoted by $(\rho_n)$) such that
		\begin{equation*}
		\partial_{x_i}\left[\lambda(|\nabla U^N_{\rho_n}|) \partial_{x_j} U^N_{\rho_n}\right] \rightharpoonup V \quad \text{weakly in } L^q(\R^3)\text{ for }n\rightarrow\infty.
		\end{equation*}
		$V$ is the weak derivative of
		\begin{equation*}
		\lambda(|\nabla U^N_{\rho}|) \partial_{x_j} U^N_{\rho}
		\end{equation*}
		with respect to $x_i$ and hence
		\begin{equation*}
		\lambda(|\nabla U^N_{\rho}|) \nabla U^N_{\rho} \in W^{1,q}_{loc}(\R^3).
		\end{equation*}
		Since the $L^q$-norm is weakly lower semi-continuous, \eqref{equ derivative of MOND correction estimate} implies
		\begin{align*}
		\left\| V \right\|_{L^q(B_R)} & \leq \liminf_{n\rightarrow\infty}\left\| \partial_{x_i}\left[\lambda(|\nabla U^N_{\rho_n}|) \partial_{x_j} U^N_{\rho_n}\right] \right\|_{L^q(B_R)} \\
		& \leq C \lim_{n\rightarrow\infty}\|\rho_n\|_p^{1/2} \\
		& = C\|\rho\|_p^{1/2}.
		\end{align*}
	\end{proof}
	
	Thus, in spherical symmetry it follows from Lemma \ref{lemma bound for derivative of lambda(nabla UM)nabla UM} (and Lemma \ref{lemma Newtonian potential for density in Lp NEW}) that the Mondian potential $U^M_\rho$ is always twice weakly differentiable.
	
	But, as we have argued in the introduction to this section, one might expect from a naive argumentation that
	\begin{equation*}
	D^2 U^M_\rho \in L^q(\R^3)
	\end{equation*}
	for $1<q<6$ if
	\begin{equation*}
	\rho \in L^1\cap L^\infty(\R^3).
	\end{equation*}
	However, in Lemma \ref{lemma bound for derivative of lambda(nabla UM)nabla UM} we were only able to prove an estimate of the type
	\begin{equation*}
	\| D^2 U^M_\rho \|_q \leq C \|\rho\|_p^{1/2}
	\end{equation*}
	if $1<q<2$. In the following Lemma we show that this estimate is indeed optimal; there is no such estimate if $q>2$. Further the subsequent Lemma will show that it is unlikely that any such estimate can be proven if we drop the assumption of spherical symmetry.
	
	\begin{lem} \label{lemma there is no bound for D2 UM in Lq with q>2}
		Let $\lambda(\sigma)=1/\sqrt{\sigma}$, $\sigma>0$. Then there is a sequence of spherically symmetric densities $(\rho_n) \subset L^1\cap L^\infty(\R^3)$ such that for all $n\in\N$ $\rho_n\geq 0$, $\supp\rho_n \subset B_2$ and $\|\rho_n\|_\infty \leq 1$, but
		\begin{equation*}
		\| D^2 U^M_{\rho_n} \|_q \rightarrow \infty \quad \text{for } n\rightarrow\infty
		\end{equation*}
		if $2<q<6$.
	\end{lem}
	
	\begin{rem}
		The idea behind the proof of Lemma \ref{lemma there is no bound for D2 UM in Lq with q>2} is the following: In $\Delta U^M_\rho$ appears the term
		\begin{equation*}
		\frac{\sqrt{M(r)}'}{r} = \frac{2\pi r\rho(r)}{\sqrt{M(r)}} = \frac{2\pi\rho(r)}{\sqrt{N(r)}} \frac{1}{\sqrt r}
		\end{equation*}
		where we have introduced the notion
		\begin{equation*}
		N(r) := \frac{1}{r^3} M(r) = \frac{1}{r^3} \int_{B_r} \rho(x) \diff x.
		\end{equation*}
		\cite[Satz 7.7]{1999Rudin} implies that for a.e. $y\in\R^3$
		\begin{equation*}
		\frac{1}{r^3} \int_{B_r(y)} \rho(x) \diff x \rightarrow \frac{4\pi}{3}\rho(y) \quad \text{for }r\rightarrow 0.
		\end{equation*}
		So we could expect that
		\begin{equation*}
		\frac{\sqrt{M(r)}'}{r} \approx \frac{\sqrt{3\pi} \rho(r)}{\sqrt{\rho(0)}} \frac{1}{\sqrt{r}} \quad \text{for }r>0 \text{ small}.
		\end{equation*}
		Assuming for the moment that $\rho(0)>0$ and that $\|\rho\|_\infty < \infty$ this would guarantee that $\|\sqrt{M(r)}'/r\|_q$ is bounded for all $1 < q < 6$. Together with Proposition \ref{prop bounds on ( sqrt(M)/r ) Prime} this would give us a bound for $\|D^2U^M_\rho\|_q$ for all $1< q < 6$. However, the pointwise representation of an $L^p$-function $\rho$ is tricky:
		
		Lets take an open set $\Omega_n\subset[0,2]$ such that for all $\epsilon>0$
		\begin{equation*}
		\measure(\Omega_n\cap[0,\epsilon]) \approx \frac{\epsilon}{n}
		\end{equation*}
		and set 
		\begin{equation*}
		\rho_n(r) := 1_{\Omega_n}(r).
		\end{equation*}
		Then there is no well defined value of $\rho(0)$ and we get
		\begin{equation*}
		N(r) \approx \frac{C}{n} \quad\text{for }r>0\text{ small}
		\end{equation*}
		with a constant $C>0$ independent of $n$.
		Thus
		\begin{equation*}
		\frac{\sqrt{M(r)}'}{r} \approx \frac{2\pi}{\sqrt C}\sqrt{n}1_{\Omega_n}(r) \frac{1}{\sqrt{r}} \quad \text{for }r>0\text{ small},
		\end{equation*}
		and when we send $n\rightarrow\infty$ this is unbounded in $L^q$ for $2<q<6$.
	\end{rem}

	\begin{proof}[Proof of Lemma \ref{lemma there is no bound for D2 UM in Lq with q>2}]
		For $n\in\N$ set
		\begin{equation*}
		\Omega_n := \bigcup_{i=0}^\infty \left[2^{-i}, \left(1 + \frac 1n\right)2^{-i} \right)
		\end{equation*}
		and define
		\begin{equation*}
		\rho_n(r) := \frac{1}{4\pi} 1_{\Omega_n}(r), \quad r\geq 0.
		\end{equation*}
		Denote by
		\begin{equation*}
		M_n(r) := \int_{B_r} \rho_n \diff x
		\end{equation*}
		the mass of $\rho_n$ inside the ball with radius $r\geq 0$. Let $n,j\in\N$, then
		\begin{align*}
		M_n(2^{-j+1}) & = \sum_{i=j}^\infty \int_{2^{-i}}^{(1+1/n)2^{-i}} r^2 \diff r \leq \sum_{i=j}^\infty \frac 1n 2^{-i}(2^{-i+1})^2 = \frac{4}{n} \sum_{i=j}^\infty \left(\frac 18 \right)^i \\
		& = \frac{4}{n} \left( \frac{1}{1-1/8} - \frac{1-(1/8)^j}{1-1/8} \right) = \frac{4}{n} \left(\frac 18 \right)^j \frac 87 = \frac{C_0}{n}(2^{-j})^3.
		\end{align*}
		Let $r\in[2^{-j},2^{-j+1})$ for a $j\geq 0$. Then
		\begin{equation*}
		M_n(r) \leq M_n(2^{-j+1}) \leq \frac{C_0}{n} (2^{-j})^3 \leq \frac{C_0}{n}r^3.
		\end{equation*}
		Thus
		\begin{equation*}
		N_n(r) := \frac{1}{r^3}M_n(r) \leq \frac{C_0}{n}
		\end{equation*}
		and
		\begin{equation} \label{equ estimate for rho/sqrt N}
		\frac{\rho_n(r)}{\sqrt{N_n(r)}} \geq \frac{\sqrt{n}}{4\pi\sqrt{C_0}}1_{\Omega_n}(r).
		\end{equation}
		Let now $2<q<6$, then
		\begin{align*}
		\|D^2 U^M_{\rho_n}\|_q & \geq C \| \Delta U^M_{\rho_n}\|_q = C\left\| 4\pi\rho_n(r) + \frac{\sqrt{M_n(r)}}{r^2} + \frac{\sqrt{M_n(r)}'}{r}\right\|_q .
		\end{align*}
		Since $\rho_n,M_n\geq0$ and $M_n$ is monotonic increasing
		\begin{align*}
		\|D^2 U^M_{\rho_n}\|_q\geq C\left\|\frac{\sqrt{M_n(r)}'}{r}\right\|_q = C \left\|\frac{r\rho_n(r)}{\sqrt{M_n(r)}}\right\|_q = C \left\|\frac{\rho_n(r)}{\sqrt{N_n(r)}} r^{-1/2}\right\|_q.
		\end{align*}
		Now we use the estimate \eqref{equ estimate for rho/sqrt N} and get
		\begin{align*}
		\|D^2 U^M_{\rho_n}\|_q &\geq C \sqrt n \left\| r^{-1/2}1_{\Omega_n}(r) \right\|_q \\
		& = C\sqrt n \left( \int_{\Omega_n} r^{2-q/2} \diff r  \right)^{1/q} \\
		& = C \sqrt n \left( \sum_{i=0}^\infty \int_{2^{-i}}^{(1+1/n)2^{-i}} r^{2-q/2}\diff r \right)^{1/q}.
		\end{align*}
		For $2<q\leq 4$ we have
		\begin{equation*}
		\sum_{i=0}^\infty \int_{2^{-i}}^{(1+1/n)2^{-i}} r^{2-q/2}\diff r \geq \sum_{i=0}^\infty \frac{1}{n} 2^{-i}(2^{-i})^{2-q/2} = \frac 1n \sum_{i=0}^\infty (2^{-3+q/2})^i
		\end{equation*}
		and for $4<q<6$
		\begin{equation*}
		\sum_{i=0}^\infty \int_{2^{-i}}^{(1+1/n)2^{-i}} r^{2-q/2}\diff r \geq \sum_{i=0}^\infty \frac 1n 2^{-i}(2^{-i+1})^{2-q/2} = \frac 1n 2^{2-q/2} \sum_{i=0}^\infty (2^{-3+q/2})^i.
		\end{equation*}
		Hence
		\begin{equation*}
		\| D^2U^M_{\rho_n} \|_q \geq Cn^{1/2-1/q},
		\end{equation*}
		and this is divergent if $q>2$.
		
	\end{proof}

	So it is not possible for any $q>2$ to prove an estimate of the form
	\begin{equation*}
	\|D^2U^M_\rho\|_{L^q(B_R)} \leq C \|\rho\|_p^{1/2}
	\end{equation*}
	even if $\rho$ is spherically symmetric (and non-negative). Will the situation get even worse if we leave spherical symmetry?
	
	Let us look at the difficulties that one can encounter. $D^2U^M_\rho$ causes difficulties when $\nabla U^N_\rho(x)=0$ for an $x\in\R^3$ because then
	\begin{equation*}
	\lambda\left(\left|\nabla U^N_\rho(x+y)\right|\right)\left|\nabla U^N_\rho(x+y)\right| =\left|\nabla U^N_\rho(x+y)\right|^{1/2} \approx C\sqrt y
	\end{equation*}
	if $|y|$ is small and $\lambda(\sigma)=1/\sqrt\sigma$ for $\sigma>0$. Consider now the following, symmetry free situation: For every $n\in\N$ place a point mass at position
	\begin{equation*}
	x_n=\left(1-1/n,0,0\right).
	\end{equation*}
	Then for every $n\in\N$ there is $0<\alpha_n<1$ such that for
	\begin{equation*}
	y_n = \alpha_n x_n + (1-\alpha_n)x_{i+1}
	\end{equation*}
	we have
	\begin{equation*}
	\nabla U^N(y_n) = 0;
	\end{equation*}
	$U^N$ denotes the Newtonian gravitational potential created by all the masses at the points $x_n$. So for every $n\in\N$ $D^2U^N(y_n)$ will cause difficulties.
	
	The exact treatment of such a non-symmetric situation is difficult. Can we perhaps mimic the above difficulties in spherical symmetry? The answer is yes, if we do not demand that $\rho$ has to be non-negative. Then the next lemma shows that it is no more possible for any $1\leq p,q \leq \infty$ to prove an estimate of the form
	\begin{equation*}
	\|D^2U^M_\rho\|_{L^q(B_R)} \leq C \|\rho\|_p^{1/2}.
	\end{equation*}
	
	\begin{lem} \label{lemma nabla UM rho not in W11 loc}
		Let $\lambda(\sigma) = 1/\sqrt\sigma$, $\sigma>0$. Then there exists a $\rho\in L^1\cap L^\infty(\R^3)$, spherically symmetric, which takes positive and negative values, such that
		\begin{equation*}
		\nabla U^M_\rho \notin W^{1,1}_{loc}(\R^3).
		\end{equation*}
	\end{lem}
	
	\begin{proof}
		For $n\in\N$ set
		\begin{equation*}
		a_n := \sum_{i=1}^n \frac{2}{i^2}
		\end{equation*}
		and let $m_n$ be the center between $a_n$ and $a_{n+1}$, i.e.,
		\begin{equation*}
		m_n :=  a_n + \frac{1}{(n+1)^2}.
		\end{equation*}
		Then $a_1=2$ and
		\begin{equation*}
		a_n\rightarrow \frac{\pi^2}{3} <4 \quad \text{for }n\rightarrow\infty.
		\end{equation*}
		Set  $M(r):= 0$ if $r\in[0,2)$ or $r\in[\pi^2/3,\infty)$. If $r\in[2,\pi^2/3)$ set
		\begin{equation*}
		M(r) := \begin{cases}
		\alpha & \text{if }r\in[a_n,m_n) \text{ and } r = a_n+\alpha \\
		1/(n+1)^2-\alpha & \text{if }r\in[m_n,a_{n+1})\text{ and } r = m_n + \alpha
		\end{cases}.
		\end{equation*}
		Then $M$ is continuous and
		\begin{align} \label{equ M(an) = 0}
		M(a_n) & = 0, \\
		M(m_n) & = \frac{1}{(n+1)^2}. \nonumber
		\end{align}
		Set $\rho(r):= 0$ if $r\in[0,2)$ or $r\in[\pi^2/3,\infty)$. If $r\in[2,\pi^2/3)$ set	
		\begin{equation*}
		\rho(r) := \begin{cases}
		1/(4\pi r^2) & \text{if }r\in[a_n,m_n) \\
		-1/(4\pi r^2) & \text{if }r\in[m_n,a_{n+1})
		\end{cases}.
		\end{equation*}
		Then $\rho\in L^1\cap L^\infty(\R^3)$. Further for $r\geq 0$
		\begin{equation*}
		M'(r) = 4\pi r^2\rho(r)
		\end{equation*}
		and thus
		\begin{equation*}
		M(r) = \int_0^r 4\pi s^2\rho(s)\diff s = \int_{B_r} \rho \diff x.
		\end{equation*}
		In view of \eqref{equ M(an) = 0}
		\begin{equation*}
		\nabla U^N_\rho (x) = \frac{M(r)}{r^2} \frac{x}{r}
		\end{equation*}
		will have a zero for all $x=(a_n,0,0)$, $n\in\N$. Let us see how this troubles the second derivatives of the Mondian potential:
		
		As in the introduction to this section we have
		\begin{equation*}
		\divergence (\nabla U^M_\rho(x)) = \rho(x) + \frac{1}{r^2}(r\sqrt{M(r)})'
		\end{equation*}
		for $r=|x|>0$. But
		\begin{equation*}
		\frac{1}{r^2}(r\sqrt{M(r)})' \notin L^1(B_4)
		\end{equation*}
		since
		\begin{align*}
		\int_{B_4} \left| \frac{1}{r^2}(r\sqrt{M(r)})' \right| \diff x & = 4\pi \int_0^4 \left| (r\sqrt{M(r)})' \right| \diff r \\
		& = 4\pi \sum_{i=1}^\infty \left[ \int_{a_n}^{m_n} (r\sqrt{M(r)})' \diff r - \int_{m_n}^{a_{n+1}} (r\sqrt{M(r)})' \diff r \right] \\
		& = 8\pi \sum_{i=1}^\infty m_n \sqrt{M(m_n)} \geq 8\pi \sum_{i=1}^\infty \frac{1}{n+1} = \infty.
		\end{align*}
		Hence
		\begin{equation*}
		\divergence(\nabla U^M_\rho) \notin L^1(B_4)
		\end{equation*}
		and
		\begin{equation*}
		\nabla U^M_\rho \notin W^{1,1}_{loc}(\R^3).
		\end{equation*}
	\end{proof}

	Since the density $\rho$ constructed in Lemma \ref{lemma nabla UM rho not in W11 loc} mimics the difficulties that one can encounter in a situation without symmetry assumptions, we suspect that it is impossible to prove the existence of weak, integrable derivatives of $\nabla U^M_\rho$ for general $\rho\in L^1 \cap L^\infty(\R^3)$, $\geq 0$. Thus the assumption of spherical symmetry in Lemma \ref{lemma bound for derivative of lambda(nabla UM)nabla UM} seems indeed to be necessary if one wants to prove that $U^M_\rho$ is twice weakly differentiable.
	
	\section{Discussion} \label{section discussion}
	
	We have conducted an extensive analysis of the QUMOND theory, focusing initially on the gradient $\nabla U^M_\rho$ of the Mondian potential instead of directly studying the potential $U^M_\rho$. Our investigation reveals that this gradient is the irrotational part of the vector field \eqref{intro MOND field first ansatz} in the sense of the Helmholtz-Weyl decomposition. This assures that $\nabla U^M_\rho$ is an $L^p$ vector field and indeed the weak gradient of a potential. Our findings show that the corresponding potential is given by the formula from \cite{2010MNRAS.403..886Milgrom} and that it is well defined.
	
	These results were attained through a careful examination of the operator $H$ responsible for extracting the irrotational part of a vector field.  We developed a new, explicit expression for this operator using singular integral operators. Using the operator $H$ also significantly aided in demonstrating that the Mondian potential solves the PDE \eqref{intro MOND field via QUMOND PDE} in distribution sense. Thus by linking the QUMOND theory with the Helmholtz-Weyl decomposition, we established a robust mathematical foundation for QUMOND.
	
	Furthermore, we investigated second-order derivatives of the Mondian potential $U^M_\rho$. Under the additional assumption of spherical symmetry, we proved that the Mondian potential is twice weakly differentiable. However, the regularity of the second derivatives was found to be weaker than anticipated. Additionally, we illustrated why proving a similarly general regularity result for the second derivatives without symmetry assumptions seems impossible.
	
	Our findings can be applied to many problems in QUMOND. For instance, in an accompanying paper \citep{2024arXiv240211043Frenkler}, we establish the stability of a large class of spherically symmetric models. The perturbations permitted are still confined to spherical symmetry and removing this restriction draws heavily upon the results presented in this paper, a discussion of which is provided in the accompanying work. Moreover, our results can be applied to analyse initial value problems. Recent work by Carina Keller in her master's thesis demonstrates the existence of global weak solutions to the initial value problem for the collisionless Boltzmann equation. Her result is limited to spherically symmetric solutions. Generalizing it to solutions devoid of symmetry restrictions necessitates the use of the theory presented here and a further generalization of it: We have to use that the operator $H$ also preserves H\"older continuity. This is work in progress.
	
	Our research contributes to the investigation of solutions to the initial value problem for the collisionless Boltzmann equation in yet another way. Building upon the theory of \cite{1989DiPernaLions}, we have established that weak Lagrangian solutions conserve energy. This unpublished result, not imposing any symmetry restrictions, heavily relies on the results proven in this paper. Further, the question of whether every Eulerian solution of the collisionless Boltzmann equation is also a Lagrangian one, and vice versa, is of considerable interest. \cite{1989DiPernaLions} have shown that this equivalence holds if the Mondian potential has second-order weak derivatives. Thus, our findings confirm this equivalence for spherically symmetric solutions, but cast doubt on extending this conclusion to nonsymmetric scenarios.
	
	In summary, with QUMOND now placed on a robust mathematical foundation, it is possible to analyse many interesting yet unsolved questions with mathematical rigour.

	\bibliographystyle{mnras}
	\bibliography{bibliography_math,bibliography_phys,bibliography_mond}

	\appendix
	
	\section{Appendix}
	We omitted the proof of Lemma \ref{lemma lambda is Hoelder continuous} and we give it now in the appendix.
	\begin{proof}[Proof of Lemma \ref{lemma lambda is Hoelder continuous}]
		Let $\sigma>0$, then
		\begin{equation*}
		\lambda(\sigma) = - \int_{\sigma}^{\infty}\lambda'(s)\diff s \leq \frac{\Lambda}{2} \int_{\sigma}^{\infty} \frac{\diff s}{s^{3/2}} = \frac{\Lambda}{\sqrt{\sigma}}
		\end{equation*}
		as desired.
		Further, the function $\lambda(|u|)u$ is continuously differentiable on $\R^3\backslash\{0\}$, and for $u\in\R^3$, $u\neq 0$, holds
		\begin{equation*}
		D(\lambda(|u|)u) = \lambda(|u|)E_3 + \lambda'(|u|)\frac{uu^T}{|u|}
		\end{equation*}
		where $E_3$ denotes the identity matrix of dimension 3. Using the bounds for $\lambda$ and $\lambda'$, we have
		\begin{equation*}
		|D(\lambda(|u|)u)| \leq \frac{C}{\sqrt{|u|}}.
		\end{equation*}
		Let now $u,v\in\R^3$ be such that for all $t\in[0,1]$
		\begin{equation*}
		w_t := v + t(u-v)
		\end{equation*}
		is different from zero. Then
		\begin{equation*}
		|\lambda(|u|)u - \lambda(|v|)v| \leq \int_{0}^{1} \left| \frac{\diff}{\diff t} (\lambda(|w_t|)w_t) \right| \diff t \leq C \int_{0}^{1} \frac{|u-v|^{1/2}}{|w_t|^{1/2}} \diff t |u-v|^{1/2}.
		\end{equation*}
		Set
		\begin{equation*}
		a := \frac{v}{|u-v|} \quad \text{and} \quad b:=\frac{u-v}{|u-v|}
		\end{equation*}
		then $|b|=1$ and we have
		\begin{equation*}
		\int_{0}^{1} \frac{|u-v|^{1/2}}{|w_t|^{1/2}} \diff t = \int_0^1 \frac{\diff t} {|a+tb|^{1/2}} \leq 2 \int_{0}^{1/2} \frac{\diff s}{\sqrt{s}} < \infty.
		\end{equation*}
		Thus for a.e. $u,v\in\R^3\backslash\{0\}$
		\begin{equation*}
		|\lambda(|u|)u - \lambda(|v|)v| \leq C|u-v|^{1/2}.
		\end{equation*}
		By continuity this holds for all $u,v\neq0$ and due to the H\"older continuity this holds for all $u,v\in\R^3$.
	\end{proof}
	
\end{document}